\documentclass[11pt, oneside]{amsart}
\usepackage{command}
\title[BV algebra structure on the Hochschild cohomology of $E_\infty$-algebras]{Batalin-Vilkovisky algebra structure on the Hochschild cohomology of $E_\infty$-algebras}
\author{Ismaïl RAZACK}
\address{LAMFA, CNRS UMR 7352, Université de Picardie Jules Verne,  33, rue Saint-Leu, 80000, Amiens, France.}
\email{ismail.razack@u-picardie.fr}
\begin{document}
\begin{abstract}
    When $\mathcal{M}$ is a smooth, oriented, compact and simply connected manifold, Luc Menichi has shown that $HH^\ast(C^\ast(\mathcal{M}; \mathbb{F}))$, the Hochschild cohomology of the singular cochain complex of $\mathcal{M}$ is a Batalin-Vilkovisky algebra. Using the properties of algebras over the Barratt-Eccles operad, we show that this results holds even when the manifold is not simply connected. Furthermore, we prove a similar result for pseudomanifolds. Namely, we explain why $HH^\ast_\bullet(\widetilde N^\ast_\bullet(X;\mathbb{F}))$, the Hochschild cohomology of the blown-up intersection cochain complex of a compact, oriented pseudomanifold $X$, is endowed with a Batalin-Vilkovisky algebra structure.
\end{abstract}
\maketitle
{\small\tableofcontents}

\section*{Introduction}
\indent The Hochschild cohomology $HH^\ast(A)$ of a differential graded algebra $A$ (DGA) over a commutative ring $R$ has the structure of a \emph{Gerstenhaber algebra} \cite{Ger63}. This structure can be enhanced into a \emph{Batalin-Vilkovisky algebra} (BV algebra) when $A$ satisfies some form of symmetry. A certain number of cases are collected in \cite{Abb15}. In particular, this is true for the singular cochain complex of certain manifolds. Luc Menichi showed the following result.
\begin{thm*}[{\cite[Theorem 22]{Men09}}]
    Let $\mathcal{M}$ be a compact, simply-connected, oriented, smooth manifold. Then, there is a Batalin-Vilkovisky algebra structure on $HH^\ast(C^\ast(\mathcal{M}; \mathbb{F}))$, the Hochschild cohomology of the singular cochains of $\mathcal{M}$ with coefficients in a field $\mathbb{F}$.
\end{thm*}
More generally, Menichi proved that $HH^\ast(A)$ is BV algebra when $A$ is a \emph{derived Poincaré duality algebra} (DPDA) i.e.~when $A$ is isomorphic to its linear dual (up to a shift in degree) in $\mathscr{D}(A^e)$, the derived category of $A$-bimodules (and verifies a certain cyclicity condition). The precise statement is recalled in Proposition \ref{prop:12_Menichi}. We quickly explain why under the assumptions of the above given theorem $C^\ast(\mathcal{M}; \mathbb{F})$ is a DPDA. Poincaré duality is satisfied since $\mathcal{M}$ is a compact, oriented, smooth manifold. Taking the cap product with a fundamental cycle $\xi_M$ gives a quasi-isomorphism of left $C^\ast(\mathcal{M}; \mathbb{F})$-modules (we omit the shift in degree)
\[\begin{array}{ccc}
     C^\ast(\mathcal{M}; \mathbb{F}) &\xlongrightarrow{\simeq} &C_\ast(\mathcal{M}, \mathbb{F})  \\
     \phi & \longmapsto & \phi\cap \xi_M. 
\end{array}\]
Since $\mathcal{M}$ is simply-connected, we have Jones' isomorphism \cite{Jon87}
\[H_\ast(\mathcal{L}\mathcal{M}; \mathbb{F}) \simeq HH^\ast(C^\ast(\mathcal{M}; \mathbb{F}), C_\ast(\mathcal{M}; \mathbb{F}))\]
where $\mathcal{L}\mathcal{M}:=\mathcal{C}^0(\mathbb{S}^1, \mathcal{M})$ is the loop space of $\mathcal{M}$. Using this isomorphism, we're able to lift the class $[\xi_M]$ into an element in $HH^\ast(C^\ast(\mathcal{M}; \mathbb{F}), C_\ast(\mathcal{M}; \mathbb{F}))$. Thus, we get a quasi-isomorphism $P\xrightarrow{\simeq} C_\ast(\mathcal{M}; \mathbb{F})$ of $C^\ast(\mathcal{M}; \mathbb{F})$-bimodules with $P$ a cofibrant approximation of $C^\ast(\mathcal{M}; \mathbb{F})$ as a bimodule. In other words, this proves that $C^\ast(\mathcal{M}; \mathbb{F})$ is a DPDA. 

In this paper, using tools from operad theory, we show that we can lift Poincaré duality in the derived category without assuming $\mathcal{M}$ to be simply connected. Our proof is purely algebraic and only relies on the multiplicative structure on the singular cochain complex. We will prove the following statement (see corollary \ref{coro:algBE_cochain_envalg}).
\begin{claim*}
Let $\mathbb{F}$ be a field. There exists a functor $E: Top^{op} \to DGA$ from topological spaces to the category of DGAs such that for any topological space $\mathcal{X}$ the following properties are verified:
\begin{enumerate}[label=\roman*)]
    \item we have a quasi-isomorphism of left $E(\mathcal{X})$-modules $E(\mathcal{X})\xtwoheadrightarrow{\simeq} C^\ast(\mathcal{X}; \mathbb{F})$,            
    \item there is a DGA morphism $C^\ast(\mathcal{X}; \mathbb{F})\otimes C^\ast(\mathcal{X}; \mathbb{F})^{op}\to E(\mathcal{X})$,
    \item and there exists an isomorphism of DGAs $E(\mathcal{X}) \simeq E(\mathcal{X})^{op}$.
\end{enumerate}
Furthermore, these morphisms induce natural transformations. 
\end{claim*}
We will show, using i) and iii), that we can construct morphisms of left $E(\mathcal{X})$-modules
\[C^\ast(\mathcal{X}; \mathbb{F}) \xleftarrow{\simeq} E(\mathcal{X}) \rightarrow C_\ast(\mathcal{X}; \mathbb{F}).\]
The morphism on the right is a quasi-isomorphism if $\mathcal{X}$ satisfies Poincaré duality. Note that $C_\ast(\mathcal{X}; \mathbb{F})$, being isomorphic to the dual of $C^\ast(\mathcal{X}; \mathbb{F})$, is only endowed with a right $E(\mathcal{X})$-module structure at first glance. Property iii) implies that it is also a left $E(\mathcal{X})$-module. By ii), the morphisms given above are morphisms of $C^\ast(\mathcal{X}; \mathbb{F})$-bimodules. We then factorize a cofibrant resolution $P\xtwoheadrightarrow{\simeq} C^\ast(\mathcal{X}; \mathbb{F})$ through $E(\mathcal{X})$ to get quasi-isomorphisms of $C^\ast(\mathcal{X}; \mathbb{F})$-bimodules
\[C^\ast(\mathcal{X}; \mathbb{F}) \xleftarrow{\simeq} P \xrightarrow{\simeq} C_\ast(\mathcal{X}; \mathbb{F}).\]

Our candidate for the algebra $E(\mathcal{X})$ is $\env(\overline{C^\ast(\mathcal{X}; \mathbb{F})})$, the enveloping algebra of $\overline{C^\ast(\mathcal{X}; \mathbb{F})}$ over \emph{the Barratt-Eccles operad} $\mathcal{E}$ (see subsection \ref{subsec:BE}). 

Using the claim mentionned above we're able to prove the following statement.

\begin{thm*}[{\textbf{\ref{thm:BV-Hoch-via-operads}}}]
    Let $\mathcal{M}$ be a compact, oriented, smooth manifold. Then, there exists a Batalin-Vilkovisky algebra structure on $HH^\ast(C^\ast(\mathcal{M}; \mathbb{F}))$, the Hochschild cohomology of the singular cochains of $\mathcal{M}$ with coefficients in a field $\mathbb{F}$.
\end{thm*}

By \cite[Theorem 3]{FMT05}, we know that a homotopy equivalence $f:\mathcal{X}_1\to \mathcal{X}_2$ between topological spaces induces an isomorphism of Gerstenhaber algebras $HH^\ast(C^\ast(\mathcal{X}_1; \mathbb{F}))\simeq HH^\ast(C^\ast(\mathcal{X}_2; \mathbb{F}))$. Furthermore, if we suppose that $\mathcal{X}_1, \mathcal{X}_2$ satisfy Poincaré duality and if $f$ sends a fundamental class of $\mathcal{X}_1$ to a fundamental class of $\mathcal{X}_2$, it is possible to show that $f$ induces an isomorphism of BV algebras $HH^\ast(C^\ast(\mathcal{X}_1; \mathbb{F}))\simeq HH^\ast(C^\ast(\mathcal{X}_2; \mathbb{F}))$. In other words, the BV algebra structure we obtain is an oriented homotopy type invariant.   

In a previous paper \cite{Raz23}, we asked ourselves if Luc Menichi's result holds for spaces with singularities, notably for \emph{pseudomanifolds}. Note that, in general, Poincaré duality is not verified for such spaces so one can't directly adapt Menichi's approach. Hopefully, Goresky and MacPherson \cite{GM80} have introduced the \emph{intersection chain complex} in order to restore Poincaré duality for pseudomanifolds. It corresponds to a family of chain complexes $\{I^{\overline{p}}C_\ast(X; R)\}_{\overline{p}\in P_{X}}$ indexed by a poset of \emph{perversities}. The perversities are geometric parameters which restrict the chains which are allowed to intersect the singular part of a pseudomanifold. Goresky and MacPherson have also presented the intersection cochain complex but one disadvantage of their approach is that Poincaré duality is only obtained as an isomorphism in homology and one can't lift this result at the level of chain complexes. There exist several variants to Goresky and MacPherson's intersection (co)homology theory. In particular, we work with the \emph{blown-up intersection cochain complex} studied by Chataur, Saralegui and Tanré in \cite{CST18BUP-Alpine} for instance. The reader can consult appendix \ref{app:inter_hom} for a quick presentation of these intersection homology theories. Using a blown-up version of Sullivan's polynomial forms, we've shown the following result. 
\begin{proposition*}[{\cite[Proposition 4.2.1]{Raz23}}]
        Let $X$ be a compact, oriented, second-countable pseudomanifold. Then, there exists a Batalin-Vilkovisky algebra structure on $HH^\ast_\bullet(\widetilde N_\bullet^\ast(X; \mathbb{Q}))$, the Hochschild cohomology of the blown-up cochain complex of a pseudomanifold $X$ with coefficients over $\mathbb{Q}$. 
\end{proposition*}

In this paper, following an approach similar to the one described above for the singular cochain complex of a manifold, we prove the next result. 

\begin{thm*}[\textbf{\ref{thm:BV-Hoch-bup}}]
        Let $X$ be a compact, oriented, second-coutable pseudomanifold. Then, there exists a perverse Batalin-Vilkovisky algebra structure on $HH^\ast_\bullet(\widetilde N_\bullet^\ast(X; \mathbb{F}))$, the Hochschild cohomology of the blown-up cochain complex of a pseudomanifold $X$ with coefficients over any field $\mathbb{F}$. 
\end{thm*}
Furthermore, by \cite[Proposition 4.2.2]{Raz23}, if there exists a stratified homotopy equivalence $f:X_1 \to X_2$ between two pseudomanifolds $X_1, X_2$ that satisfy the assumptions of the above theorem and if $f$ sends a fundamental class on a fundamental class, then we have an isomorphism of perverse BV algebras $HH^\ast_\bullet(\widetilde N_\bullet^\ast(X_1; \mathbb{F}))\simeq HH^\ast_\bullet(\widetilde N_\bullet^\ast(X_2; \mathbb{F}))$. Hence, we have an oriented stratified homotopy type invariant which is finer than the classical Hochschild cohomology. 
\begin{center}
    \textbf{Outline of the paper}
\end{center}

In appendix \ref{app:inter_hom}, we define pseudomanifolds and we recall results about Goresky and MacPherson's intersection homology theory. We also present Chataur, Saralegui and Tanré's blown-up intersection cochain complex and give its main properties. We'll treat these objects as \emph{perverse chain complexes} and use a framework introduced by Hovey in \cite{Hov09}. These results will be used in the second section. 

In section \ref{sect:sing_cochain}, we state some general results about operads. We recall how operads and algebras over operads are defined and we present the enveloping algebra of an algebra over an operad. We also give results about the Barratt-Eccles operad $\mathcal{E}_+$. We show that the claim mentioned above follows from a general result about algebras over $\mathcal{E}_+$ (see Proposition \ref{prop:algBE_envalg}). Using this proposition, we prove Theorem \ref{thm:BV-Hoch-via-operads}.

In section \ref{sect:bup_cochain}, we first recall the construction of Hochschild cohomology for a perverse differential graded algebra and state some of its properties. We then prove that the blown-up intersection cochain complex is endowed with a structure of algebra over the Barratt-Eccles. We finally deal with the demonstration of Theorem \ref{thm:BV-Hoch-bup} by adapting to perverse chain complexes the proof of Theorem \ref{thm:BV-Hoch-via-operads}.

\begin{center}
    \textbf{Acknowledgments}
\end{center}
This preprint is based on work done during my PhD thesis on the \emph{Hochschild cohomology of intersection algebras}. I would like to thank David Chataur, my PhD thesis advisor, for valuable discussions, suggestions and for taking the time to read various drafts of this paper. I am also grateful to Benoit Fresse for a discussion about the Barratt-Eccles operad and its properties. The proof of Theorem \ref{thm:BV-Hoch-via-operads} has been simplified thanks to his help. 

\textbf{Notations and conventions:}
\begin{enumerate}[label=\roman*)]
    \item $\mathbb{F}$ denotes a field and $R$ denotes a commutative ring.
    \item Let $M$ be an $R$-module, its linear dual is denoted $M^\vee:=\Hom_R(M,R)$.
    \item $Ch(R)$ denotes the category of chain complexes. A chain complex $(Z_\ast, d)$ is a sequence of lower $\mathbb{Z}$-graded $R$-modules $\{Z_i\}_{i\in\mathbb{Z}}$ with a degree $-1$ differential. The chain complex whose components are all trivial is denoted $0$.
    \item A cochain complex $(Z^\ast, d)$ is a sequence of upper $\mathbb{Z}$-graded $R$-modules $\{Z^i\}_{i\in\mathbb{Z}}$ with a degree $+1$ differential. With the \emph{classical convention} of \cite[pp 41-42]{FHT01}, we define a chain complex $(Z_\ast, d)$ by setting for $i\in\mathbb{Z}$
    \[Z_i:=Z^{-i}.\]
    Hence, we can think of a cochain complex as an object in $Ch(R)$.
    \item Let $(Z_\ast, d)$ be a chain complex. Its \emph{linear dual} $DZ_\ast$ is the chain complex whose degree $i\in\mathbb{Z}$ component is $DZ_i:=\Hom_R(Z_{-i}, R)$. The differential is given by the precomposition with $d$. 
     \item Let $n\in \mathbb{Z}$. The $n$-\emph{sphere chain complex} $\mathbb{S}^n$ is defined as the chain complex concentrated in degree $n$ with $(\mathbb{S}^n)_n=R$ with trivial differential.\\
    The $n$-\emph{disk chain complex} is denoted $\mathbb{D}^n$. It is the chain complex whose components are $R$ in degrees $n$ and $n-1$ and $0$ elsewhere. The only non-trivial differential is $(\mathbb{D}^n)_n\to (\mathbb{D}^n)_{n-1}$ and it is given by the identity.  
    \item The \emph{suspension} $sC_\ast$ of a chain complex $C_\ast$ is the shift in degree by 1 i.e. for $n\in \mathbb{N}$, $(sC)_n=C_{n-1}$. For a cochain complex $C^\ast$, this corresponds to a shift in degree by -1 i.e. $(sC)^n=C^{n+1}$.
    \item Let $A$ be a unital algebra. It is endowed with a morphism of algebras $\eta:\mathbb{F}\to A$. We denote by $\overline{A}:=A/\eta(\mathbb{F})$ the \emph{reduced algebra} i.e. the non unital algebra associated to $A$.
\end{enumerate}

\section{Generalization of Luc Menichi's result}\label{sect:sing_cochain}
    The goal of this section is to give another proof of \cite[Theorem 22]{Men09} using operadic arguments and without assuming that the manifold is simply-connected. 
    \begin{letterthm}\label{thm:BV-Hoch-via-operads}
            Let $\mathcal{M}$ be a compact, oriented, smooth manifold. Then, there exists a Batalin-Vilkovisky algebra structure on $HH^\ast(C^\ast(\mathcal{M}; \mathbb{F}))$, the Hochschild cohomology of the singular cochains of $\mathcal{M}$ with coefficients in a field $\mathbb{F}$.
    \end{letterthm}
    In the second section of this paper, we will extend this result to $HH^\ast_\bullet(\widetilde N_\bullet^\ast(X; \mathbb{F}))$, the Hochschild cohomology of the blown-up cochain complex of a pseudomanifold $X$ with coefficients over any field $\mathbb{F}$. 
    
    \subsection{General results about operads}\label{sect: }
    
    Operads encode the different types of algebras and are useful to study properties (for instance, associativity and commutativity) up to higher homotopies. The term \emph{operad} first appeared in Peter May's study of iterated loop spaces \cite{May72}. Independently, Boardman, Vogt have considered an analogous construction in their work \cite{BV73}. The reader may consult the first chapter of Markl, Shnider, and Stasheff's book \cite{MSS02} for a detailed history of this concept. 
    
    Informally, an operad in $R$-modules is a family of modules $\{M(k)\}_{k\in \mathbb{N}}$ which represent the algebraic structures associated to a certain type of algebra. In particular, the elements of $M(k)$ contain all the data regarding the operations with $k$ arguments. We can define operads in any symmetric monoidal category. The main cases we will treat are the categories of sets, $R$-modules, chain complexes and perverse chain complexes over $R$. 
        \subsubsection{Operads and algebras over operads}
    We begin by introducing some notations for the symmetric groups, following \cite{BF04}.
    \begin{defi}
        For $k\in\mathbb{N}$, the \emph{symmetric group} on $k$ elements is the group of bijection between $\{1, \ldots, k\}$ and itself, it is denoted $\mathfrak{S}_k$. We will refer to a permutation $\sigma\in\mathfrak{S}_k$ by giving the sequence of its values $(\sigma(1), \ldots, \sigma(k))$. The identity permutation of $\mathfrak{S}_k$ is denoted $1_k$.
    \end{defi} 
    \begin{defi}
    We consider a permutation $\sigma\in\mathfrak{S}_k$. A sequence of integers $i_1, \ldots, i_k\in \mathbb{N}$ determines a partition of the set $\{1, \ldots, i_1+\ldots+i_k\}$ into $k$ blocks
    \[\{\boxed{1, \ldots, i_1}, \ldots, \boxed{i_1+\ldots+i_{k-1}+1, \ldots, i_1+\ldots+i_k}\}\]
    where for $1\leq s\leq k$, the $s^{\text{th}}$ block contains $i_s$ elements.
    The \emph{block permutation} $\sigma_\ast(i_1, \ldots, i_k)\in \mathfrak{S}_{i_1+\ldots+i_k}$ is the transformation that permutes these $k$ blocks just as $\sigma$ permutes the set $\{1, \ldots, k\}$.   
    \end{defi}
    \begin{exemple}
        If $\sigma=(3, 2, 1)$ then for $p,q,r \in \mathbb{N}$ we have $\sigma_\ast(p, q, r)=(p+q+1, \ldots, p+q+r, p+1, \ldots, p+q, 1, \ldots, p)$.
    \end{exemple}
    We can know define the notion of operad. We follow May's original approach \cite{May72}. 
    \begin{defi}
        Let $(\mathcal{C}, \otimes, \mathcal{I})$ be a symmetric monoidal category. An \emph{operad} $\mathcal{O}$ in $\mathcal{C}$ is a family $\{\mathcal{O}(k)\}_{k\in\mathbb{N}}$ of objects in $\mathcal{C}$ that is endowed with the following data (known as \emph{May's axioms}).
        \begin{itemize}
            \item \textbf{Composition products:} For any $k, i_1,\ldots, i_k\in \mathbb{N}$, we have morphisms in $\mathcal{C}$
            \[\begin{array}{rrcl}
                \gamma:& \mathcal{O}(k)\otimes\mathcal{O}(i_1)\otimes \ldots \otimes \mathcal{O}(i_k) &\to& \mathcal{O}(j), \\
                ~ & (a, b_1, \ldots, b_k) & \mapsto & \gamma(a; b_1,\ldots, b_k)
            \end{array}\]
            where $j=\sum_{s=1}^k i_s$, which verify the following associativity property: for any $a\in \mathcal{O}(k)$, $b_s\in \mathcal{O}(i_s)$, $c_r\in \mathcal{O}(i_r)$ with $1\leq s\leq k$ and $1\leq r\leq j$ we have
            \[\gamma(\gamma(a; b_1,\ldots, b_k); c_1,\ldots, c_j)=\gamma(a; d_1, \ldots, d_k)\]
            with $d_s=\gamma(b_s; c_{i_1+\ldots+i_{s-1}+1}, \ldots, c_{i_1+\ldots+i_s})$. 
            \item \textbf{Identity element:} An element $1\in \mathcal{O}(1)$ such that for any $k\in\mathbb{N}$, $a\in\mathcal{O}(k)$ we have
            \[\gamma(1; a)=a \text{ and }\gamma(a; \underbrace{1, \ldots, 1}_{k\text{ times}})=a.\]            
            \item \textbf{Equivariance relations:} For any $n\in \mathbb{N}$, the symmetric group $\mathfrak{S}_n$ acts on $\mathcal{O}(n)$ and for any $k, i_1,\ldots, i_k\in \mathbb{N}$, $\sigma\in \mathfrak{S}_k$, $\tau_s\in \mathfrak{S}_{i_s}$, $a\in \mathcal{O}(k)$, $b_s\in \mathcal{O}(i_s)$ with $1\leq s\leq k$ the following equivariance identities are satisfied
            \[\gamma(\sigma.a; b_1, \ldots, b_k)=\sigma_\ast(i_1,\ldots, i_k).\gamma(a; b_{\sigma^{-1}(1)}, \ldots, b_{\sigma^{-1}(k)}) \]
            and
            \[\gamma(a; \tau_1.b_1, \ldots, \tau_k.b_k)=(\tau_1\oplus\ldots\oplus\tau_k).\gamma(a; b_1, \ldots, b_k) \]
            where $\sigma_\ast(i_1,\ldots, i_k)$ is the block permutation associated to $\sigma$ and $\tau_1\oplus\ldots\oplus \tau_k$ denotes the image of $(\tau_1, \ldots, \tau_k)$ by the inclusion $\mathfrak{S}_{i_1}\times \ldots \times \mathfrak{S}_{i_k}$ in $\mathfrak{S}_j$. 
        \end{itemize}
    \end{defi}
    \begin{rem}
        An element $a\in \mathcal{O}(k)$ is said to be of \emph{arity} $k$ and should be thought as a multi-linear operation in $k$ variables $x_1, \ldots, x_k$. This is why we will denote it by $a(x_1,\ldots, x_k)$. In the same way, we denote the composition product $\gamma(a; b_1,\ldots, b_k)$ by $a(b_1, \ldots, b_k)$.
    \end{rem} 
    \begin{rem}
        For $k, j\in \mathbb{N}$, we can define \emph{partial composition products}, morphisms in $\mathcal{C}$ that are defined for $1\leq s\leq k$ by
        \[
        \begin{array}{rrcl}
            - \circ_s -:&   \mathcal{O}(k)\otimes \mathcal{O}(j)&\to& \mathcal{O}(k+j-1)   \\
             ~ &            (a, b) & \mapsto & a( \underbrace{1,\ldots, 1}_{s-1\text{ terms}}, b, 1, \ldots, 1). 
        \end{array}
        \]
    \end{rem}
    \begin{rem}
        The conditions that are verified by an operad can be expressed using trees, see Figures \ref{fig:asso_comp_prod}, \ref{fig:operad_identity}. One can check Ginzburg and Kapronov's article \cite{GK94} for a detailed presentations on trees.
    \end{rem}      
    
    \begin{figure}[ht]
    \centering
    \def\svgscale{0.9}
    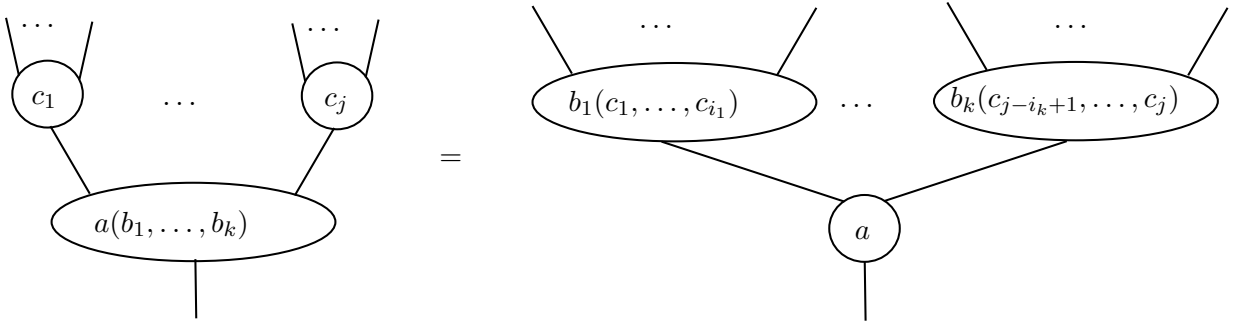
    \caption{Associativity of composition products}
    \label{fig:asso_comp_prod}
    \end{figure}
    \begin{figure}[ht]
    \centering
    \def\svgscale{0.8}
    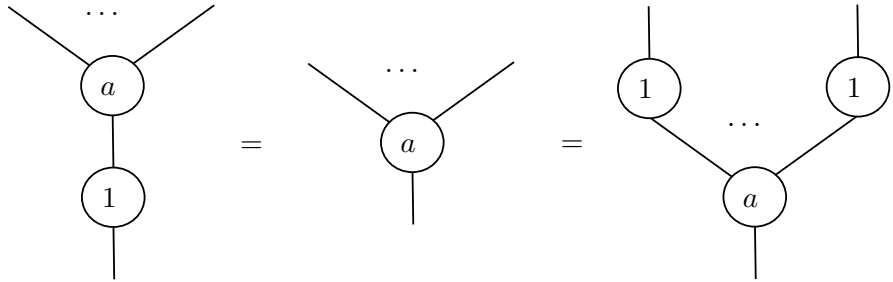
    \caption{Identity element}
    \label{fig:operad_identity}
    \end{figure}
    \begin{rem}\label{rem:sym_objects}
    \begin{itemize}
        \item There exists another way of defining an operad by considering \emph{$\Sigma_\ast$-objects}. A $\Sigma_\ast$-object in $\mathcal{C}$ is a family of objects $\{M(k)\}_{k\in\mathbb{N}}$ such that for any $k\in\mathbb{N}$, the symmetric group $\mathfrak{S}_k$ acts on $M(k)$. One can show that the category of $\Sigma_\ast$-objects is actually symmetric monoidal. An operad is then defined as a monoid in this category. More details can be found in \cite[Chapter 2]{Fre09}. 
        \item There exists a cofibrantly generated model category structure on $\Sigma_\ast$-objects and a semi-model category for operads over a model category $\mathcal{C}$. We don't give explicit descriptions of these structures but we refer to \cite[Part III]{Fre09} for constructions and details. 
    \end{itemize}
    \end{rem}
    We now give examples of operads in the category of sets.. 
    \begin{exemple}
        The \emph{commutative operad} $Com$ is the operad which is given for $k\in\mathbb{N}^\ast$ by $Com(k)=\{\ast\}$ the one-point set on which $\mathfrak{S}_k$ acts trivially. We set $Com(0)=\emptyset$.
    \end{exemple}
    \begin{exemple}
        The \emph{associative operad} $As$ is the operad which is given for $k\in\mathbb{N}^\ast$ by $As(k)=\mathfrak{S}_k$ with the action of $\mathfrak{S}_k$ given by left translation. We set $As(0)=\emptyset$.
    \end{exemple}    
    \begin{rem}\label{rem:unitary_ops}
        The associative and commutative operads are examples of \emph{non-unitary operads}. Following Fresse (\cite[Section 3.1.10]{Fre09}), we say that an operad $\mathcal{O}$ in $\mathcal{C}$ is \emph{non-unitary} if $\mathcal{O}(0)=0_{\mathcal{C}}$ the initial object of $\mathcal{C}$. An operad is \emph{unitary} if $\mathcal{O}(0)=\mathcal{I}$ where $\mathcal{I}$ is the monoidal unit of $\mathcal{C}$. There exist unitary versions of the commutative and associative operads, we denote them $Com_{+}$ and $As_{+}$ respectively. They are defined like their non-unitary counterparts but we set $Com_{+}(0)=As_{+}(0)=\{\ast\}$ the one-point set. 
    \end{rem}
    \begin{rem}
        Let $\mathcal{C}$ be a symmetric monoidal category. The forgetful functor $\mathcal{C}\to Sets$ admits a left adjoint $Sets\to \mathcal{C}$. Using it, one can define commutative and associative operads in any symmetric monoidal category. The definitions of theses operads in the category of $R$-modules is presented in \cite[Section 3.1.8]{Fre09}. For instance, the commutative operad on $R$-modules is given by  $Com(n)=R$ for any $n\in\mathbb{N}^\ast$ and $Com(0)$ is the zero module. 
    \end{rem}
    \begin{exemple}
        The \emph{permutation operad} $\mathfrak{S}$ is the operad in the category of sets (endowed with the cartesian product) whose arity $k\in \mathbb{N}$ component is given by $\mathfrak{S}(k):=\mathfrak{S}_k$. For $k, i_1, \ldots, i_k \in \mathbb{N}$ and $\sigma\in \mathfrak{S}(k)$, $\tau_s\in\mathfrak{S}(i_s)$ for $1\leq s \leq k$, the composition product is given by
        \[\sigma(\tau_1,\ldots, \tau_k):=(\tau_1\oplus\ldots\oplus\tau_k).\sigma_\ast(i_1,\ldots, i_k)\]
        which is also equal to $\sigma_\ast(i_1,\ldots, i_k).(\tau_{\sigma(1)}\oplus\ldots\oplus\tau_{\sigma(k)})$. For instance, we have 
        \[(3,2,1)((12), (12), 1_1)=(5,3,4,1,2).\]
    \end{exemple}
    It is natural to consider \emph{morphisms of operads}. They will make it possible to transfer structures between operads.
    \begin{defi}
        Let $\mathcal{O}=\{\mathcal{O}(k)\}_{k\in\mathbb{N}}$ and $\mathcal{P}=\{\mathcal{P}(k)\}_{k\in\mathbb{N}}$ be two operads in a symmetric monoidal category $\mathcal{C}$. A \emph{morphism of operads} is a family of morphisms in $\mathcal{C}$, $\{f_k:\mathcal{O}(k)\to \mathcal{P}(k)\}_{k\in\mathbb{N}}$ such that each $f_k$ commutes with the action of $\mathfrak{S}_k$ and with the composition products. 
    \end{defi} 

    We now define the notion of \emph{algebra over an operad} following \cite[Definition 2.1]{KM95}. This will correspond to objects which are endowed with the operations encoded by an operad. 
    \begin{defi}
        Let $\mathcal{O}$ be an operad in a symmetric monoidal category $(\mathcal{C}, \otimes, \mathcal{I})$. An \emph{algebra over the operad} $\mathcal{O}$ is an object $A\in \mathcal{C}$ endowed, for every $k\in\mathbb{N}$, with a morphism in $\mathcal{C}$, $\lambda:\mathcal{O}(k)\otimes A^{\otimes k}\to A$ known as \emph{evaluation morphisms} that verify the following properties.
        \begin{itemize}
            \item \textbf{Associativity:} For any $k, i_1, \ldots, i_k\in\mathbb{N}$ and $j=\sum_{s=1}^k i_s$, the following diagram commutes
            \[
            \begin{tikzcd}
            \mathcal{O}(k)\otimes\mathcal{O}(i_1)\otimes \ldots \otimes \mathcal{O}(i_k)\otimes A^{\otimes j} \arrow[rr, "\gamma\otimes\id"] \arrow[dd]                                                  &  & \mathcal{O}(j)\otimes A^{\otimes j} \arrow[d, "\lambda"]   \\
                                                                                                                                                                                    &  & A                                                          \\
            \mathcal{O}(k)\otimes \mathcal{O}(i_1)\otimes A^{\otimes i_1}\otimes \ldots \otimes \mathcal{O}(i_k)\otimes A^{\otimes i_k} \arrow[rr, "\id\otimes\lambda^{\otimes k}"'] &  & \mathcal{O}(k)\otimes A^{\otimes k}. \arrow[u, "\lambda"']
            \end{tikzcd}
            \]
            \item \textbf{Unit:} For any $a\in A$, we have $\lambda(1; a)=a$.
            \item \textbf{Equivariance:} For any $k\in \mathbb{N}$, $p\in\mathcal{O}(k)$ and $a_1,\ldots, a_k \in A$ we have the equality
            \[\lambda(p, a_1, \ldots, a_k)=\lambda(\sigma.p, a_{\sigma^{-1}(1)}, \ldots, a_{\sigma^{-1}(k)}).\]
        \end{itemize}
        In what follows, we will use the notation $p(a_1, \ldots, a_k):=\lambda(p; a_1, \ldots, a_k)$.
    \end{defi}
    \begin{rem}\label{rem:alg_op_gen_case}
         Let $\mathcal{D}$ be a \emph{symmetric monoidal category over }$\mathcal{C}$ (this means that $\mathcal{D}$ is a symmetric monoidal category and that it has a compatible external product $\mathcal{C}\times \mathcal{D}\to \mathcal{D}$). If $\mathcal{O}$ is an operad in $\mathcal{C}$, it is possible to endow an object in $\mathcal{D}$ with a structure of algebra over $\mathcal{O}$. The reader can check \cite[Chapters 1 and 3]{Fre09} for more details. 
    \end{rem}
    \begin{exemple}
        Let $\mathcal{C}$ be a symmetric monoidal category. The algebras over $Com$ (respectively $As$) in $\mathcal{C}$ are equivalent to the commutative (respectively associative) monoids in $\mathcal{C}$ without units. As for the unitary versions $Com_{+}$ and $As_{+}$, we can show that their algebras are monoids with units which are respectively commutative and associative. We refer to \cite[Section 3.2.4]{Fre09} for details. 
    \end{exemple}
    \begin{rem}\label{rem: op morph induce morph of algebras}
        Notice that an operad morphism  $\phi: \mathcal{O} \to \mathcal{P}$ induces a morphism $\phi^\ast$ between the algebras over $\mathcal{P}$ and the algebras over $\mathcal{O}$. 
    \end{rem}
    
        \subsubsection{Enveloping algebra}
    We now present the notion of \emph{enveloping algebra} $\mathcal{U}_{\mathcal{O}}(A)$ for an algebra $A$ over an operad $\mathcal{O}$ which extends the universal enveloping algebra which is known in Lie theory. Just as algebras over operads correspond to classical types of algebras, the left $\mathcal{U}_{\mathcal{O}}(A)$-modules correspond to the \emph{representations} of $A$. This is made precise in \cite[Proposition 4.3.2]{Fre09}. 

    In this subsection, we consider $\mathcal{O}$ an operad and $A$ an algebra over $\mathcal{O}$. In order to define the enveloping algebra, we will need some notations. 
    \begin{defi}
        Let $m\in \mathbb{N}$. The \emph{shifted object} $\mathcal{O}[m]$ associated to $\mathcal{O}$ is a family of chain complexes $\{\mathcal{O}[m](k)\}_{k\in \mathbb{N}}$ defined by setting $\mathcal{O}[m](k):=\mathcal{O}(m+k)$ for $k\in \mathbb{N}$. The symmetric group $\mathfrak{S}_k$ acts on $\{m+1, \ldots, m+k\}\subset \{1, \ldots, m+k\}$, hence we have an injection $\mathfrak{S}_k\to \mathfrak{S}_{k+m}$ which induces an action of $\mathfrak{S}_k$ on $\mathcal{O}[m](k)$.
    \end{defi}
    \begin{rem}
       The shifted object is actually an operad, the reader can refer to \cite[Section 4.1.5]{Fre09} for a detailed description of the composition morphisms. 
    \end{rem}
    For $k\in \mathbb{N}$, the symmetric group $\mathfrak{S}_k$ acts on an element $(p, a_1, \ldots, a_k) \in \mathcal{O}(k)\otimes A^{\otimes k}$ by $\sigma.(p, a_1, \ldots, a_k)=(\sigma.p, a_{\sigma^{-1}(1)}, \ldots, a_{\sigma^{-1}(k)})$. We set
    \[S(\mathcal{O},A):=\oplus_{k\geq 0}(\mathcal{O}(k)\otimes A^{\otimes k})_{\mathfrak{S}_k}\]
    where $(\mathcal{O}(k)\otimes A^{\otimes k})_{\mathfrak{S}_k}$ denotes the coinvariants of $\mathcal{O}(k)\otimes A^{\otimes k}$ under the action of the symmetric group given above. In particular, we have the following identification $(\sigma.p, a_1, \ldots, a_k)\equiv(p, a_{\sigma(1)}, \ldots, a_{\sigma(k)})$.
    \begin{defi}
         The \emph{enveloping algebra} $\mathcal{U}_{\mathcal{O}}(A)$ of the $\mathcal{O}$-algebra $A$ is a monoid in $\mathcal{C}$ obtained by considering a certain quotient of $S(\mathcal{O}[1], A)$. More precisely, it is spanned by formal elements $p(t, a_1, \ldots, a_k)$ with $k\in\mathbb{N}$, $a_1, \ldots, a_k\in A$, $p\in \mathcal{O}(k+1)$ and $t$ a variable with the following identification
         \[p\circ_{1+s}q(t, a_1, \ldots, a_{k+l-1})\equiv p(t, a_1, \ldots, a_{s-1}, q(a_s, \ldots, a_{s+l-1}), a_{s+l}, \ldots, a_{k+l-1})\]
         for $l\in \mathbb{N}$ and $q\in \mathcal{O}(l)$. The multiplication is given for $p\in \mathcal{O}(k+1)$ and $q\in \mathcal{O}(l+1)$ by
         \[p(t, a_1, \ldots a_k).q(t, b_1, \ldots, b_l):=p\circ_1 q(t, b_1, \ldots, b_l, a_1, \ldots, a_k).\]
    \end{defi}
    For a more detailed presentation of the enveloping algebra, the reader may refer to \cite[Section 4.3]{Fre09}. The following examples are taken from that monograph.  
    \begin{exemple}
        \noindent\begin{itemize}
            \item Let $A$ be a commutative algebra. The enveloping algebra $\mathcal{U}_{Com}(\overline{A})$ is isomorphic to $A$.
            \item  Let $A$ be an associative algebra. The enveloping algebra $\mathcal{U}_{As}(\overline{A})$ is isomorphic to $A\otimes A^{op}$.        
            We just give an explicit description of the isomorphism. It is defined by
            \[\begin{array}{ccc}
                \mathcal{U}_{As}(\overline{A}) & \to &  A\otimes A^{op}\\
                 \sigma(t, a_1, \ldots, a_k)& \mapsto & \sigma(1).\sigma(2)\ldots\sigma(i-1)\otimes \sigma(i+1)\ldots\sigma(k) 
            \end{array}\]
            with $i=\sigma^{-1}(1)$. The inverse map is given for $a,b\in \overline{A}$ by
            \[\begin{array}{ccc}
                 A\otimes A^{op} & \to & \mathcal{U}_{As}(\overline{A})\\
                 1\otimes a & \mapsto & (1, 2)(t, a) \\
                 a\otimes 1 & \mapsto & (2, 1)(t, a) \\
                 a\otimes b & \mapsto & (2, 1, 3)(t, a, b).
            \end{array}\]   
        \end{itemize}
    \end{exemple}

    The enveloping algebra is actually a functorial construction. The following result can be deduced from \cite[Section 10.1.5]{Fre09}.
    \begin{proposition}\label{prop:functo_envelop}
        Let $\phi:\mathcal{O}\to \mathcal{P}$ be a morphism of operads. For any $\mathcal{P}$-algebra $A$, $\phi$ induces a morphism of $\mathcal{O}$-algebras between the enveloping algebras:
        \[\mathcal{U}_{\mathcal{O}}(\phi^\ast(A))\to \mathcal{U}_{\mathcal{P}}(A) \]
        where $\phi^\ast$ is given in Remark \ref{rem: op morph induce morph of algebras}. 
    \end{proposition}

    We give two lemmas which will be useful later.
    \begin{lemma}\label{lem:mod_struc_uo_oalg}
        Let $A$ be an algebra over an operad $\mathcal{O}$. There exists a left $\mathcal{U}_{\mathcal{O}}(A)$-module structure on A which is given by
        \[\begin{array}{lccc}
            \mu:&\mathcal{U}_{\mathcal{O}}(A)\otimes A & \to & A \\
             &\sigma(t, a_1, \ldots, a_k)\otimes a& \mapsto & \sigma(a, a_1, \ldots, a_k).  
        \end{array}\]
    \end{lemma}    
    \begin{figure}[!h]
    \centering
    \def\svgscale{0.5}
    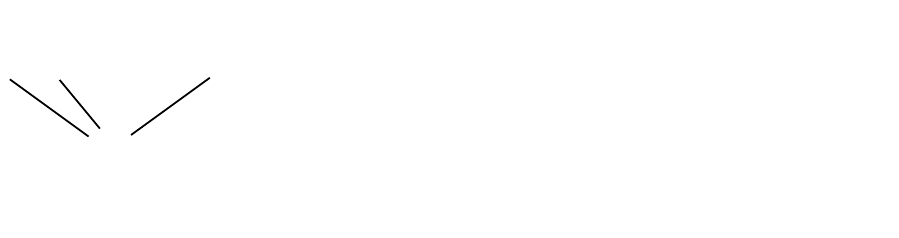
    \label{fig:module_env}
    \end{figure}
    \begin{proof}
        Let $\sigma(t, a_1, \ldots, a_k), \omega(t, b_1, \ldots, b_l)$ and $a\in A$. We want to show that
        \[(\omega(t, b_1, \ldots, b_l).\sigma(t, a_1, \ldots, a_k)).a=\omega(t, b_1, \ldots, b_l).(\sigma(t, a_1, \ldots, a_k).a).\]
        On one side, we have
        \[(\omega(t, b_1, \ldots, b_l).\sigma(t, a_1, \ldots, a_k)).a=\omega\circ_1\sigma(a, a_1, \ldots, a_k, b_1, \ldots, b_l)\]
        and on the other side 
        \[\omega(t, b_1, \ldots, b_l).(\sigma(t, a_1, \ldots, a_k).a)=\omega(\sigma(a, a_1, \ldots, a_k), b_1, \ldots, b_l).\] Hence, by the associativity of the the evaluation products, we get the identity we wanted.
    \end{proof}
    One easily proves the next lemma and deduces the corollary given below. 
    \begin{lemma}\label{lem:morph_uo_char_1}
        Let $A$ be an algebra over an operad $\mathcal{O}$. A morphism $f:\mathcal{U}_{\mathcal{O}}(A)\to A$ of $\mathcal{U}_{\mathcal{O}}(A)$-modules is characterized by the image of $1_1(t)$.
    \end{lemma}

    \begin{corollaire}\label{coro:morph_env_alg}
    Let $A$ be an algebra over an operad $\mathcal{O}$. The map
        \[\begin{array}{ccc}
            \mathcal{U}_{\mathcal{O}}(A) & \to & A \\
             \sigma(t, a_1, \ldots, a_k)& \mapsto & \sigma(1_A, a_1, \ldots, a_k).  
        \end{array}\]
    is a morphism of $\mathcal{U}_{\mathcal{O}}(A)$-modules.
    \end{corollaire}
    \subsection{The Barratt-Eccles operad}\label{subsec:BE}
        We now present the \emph{Barratt-Eccles operad} $\mathcal{E}_+$ (unitary version). In \cite{BE74}, in order to study loop spaces, Barratt and Eccles introduced a simplicial operad. By taking the normalized cochain complex associated to this simplicial operad, Berger and Fresse \cite{BF04} define the Barratt-Eccles operad. The reader can also consult \cite[Appendix A]{FG20} for a brief presentation of this operad and its properties. 
        
        \begin{defi}
        For $k\in\mathbb{N}$, we denote by $\mathcal{E}_+(k)$ the chain complex whose degree $l\in\mathbb{N}$ component $\mathcal{E}_+(k)_l$ is the $\mathbb{F}$-vector space generated by elements $\sigma_0\otimes \ldots\otimes \sigma_l$ where $\sigma_i\in \mathfrak{S}_k$ for $0\leq i \leq l$ and such that two consecutive permutations are not equal. Its differential is given by
        \[d(\sigma_0\otimes \ldots\otimes \sigma_l)=\sum_{i=0}^l (-1)^i \sigma_0\otimes \ldots\otimes \sigma_{i-1} \otimes \sigma_{i+1}\ldots\otimes \sigma_l.\]
        The action of a permutation $\sigma\in \mathfrak{S}_k$ on $\sigma_0\otimes \ldots\otimes \sigma_l$ is given by
        \[\sigma.(\sigma_0\otimes \ldots\otimes \sigma_l)=\sigma.\sigma_0\otimes \ldots\otimes \sigma.\sigma_l.\]
        The \emph{non-unitary Barratt-Eccles operad} $\mathcal{E}$ is defined by taking the trivial chain complex for $\mathcal{E}(0)$.
        
        The composition products are induced by those given on the permutation operad, details can be found in \cite[Section 1.1]{BF04}. For $i\in\mathbb{N}$, the $i$-partial composition product of $\sigma=(\sigma_0, \ldots, \sigma_d)\in\mathcal{E}_+(k)_d$ and $\tau=(\tau_0, \ldots, \tau_e)\in\mathcal{E}_+(l)_e$ is an element in $\mathcal{E}_+(k+l-1)_{d+e}$ given by
        \[\sigma\circ_i\tau:=\sum_{(x_\ast, y_\ast)}\pm (\sigma_{x_0}\circ_i \tau_{y_0}, \ldots, \sigma_{x_{d+e}}\circ_i \tau_{y_{d+e}})\]
        where the sum ranges over the set of paths from $(0,0)$ to $(d,e)$ in a $d\times e$ diagram. The sign associated to a path $\gamma$ in a $d\times e$ diagram is just the signature of the permutation in $\mathfrak{S}_{d+e-1}$ which sends $\gamma$ to the path
        \[((0,0), (1,0), \ldots, (d, 0), (d, 1), \ldots, (d, e)).\]
        \end{defi}
    One can show that there exists a quasi-isomorphism between the Barratt-Eccles operad $\mathcal{E}$ and the commutative operad $Com$. Actually, $\mathcal{E}$ is an example of an \emph{$E_\infty$-operad}, see \cite[Section 1.1]{BF04} for more information. Recall that there is a model category structure for $\Sigma_\ast$-modules (see Remark \ref{rem:sym_objects}).
    \begin{defi}\label{def:E_inf_op}
        An \emph{$E_\infty$-operad} is an operad which is a cofibrant approximation of $Com$ in the category of $\Sigma_\ast$-modules, more precisely, it is $\Sigma_\ast$-cofibrant operad $E$ endowed with an acyclic fibration $E\xtwoheadrightarrow{\simeq} Com$. An \emph{$E_\infty$-algebra} is an algebra over some $E_\infty$-operad.
    \end{defi}
    \subsubsection{Algebras over the Barratt-Eccles operad}
    The following result is proved in Berger and Fresse's article \cite{BF04}.
    \begin{thm}[{\cite[Theorem 2.1.1]{BF04}}]
        Let $K$ be a simplicial set. The normalized cochain complex associated to $K$, $N^\ast(K)$ is endowed with an $\mathcal{E}_+$-algebra structure that is functorial in $K$.
    \end{thm}
    We get the following corollary by considering the chain map between the normalized cochain complex of a topological space and its singular cochain complex.
    \begin{corollaire}\label{coro:BE alg on norm cochains}
        Let $\mathcal{X}$ be a topological space. The singular cochain complex associated to $\mathcal{X}$, $C^\ast(\mathcal{X}; \mathbb{F})$ is endowed with an $\mathcal{E}_+$-algebra structure.
    \end{corollaire}
    We will need the next proposition to lift Poincaré duality in the derived category. 
    \begin{proposition}\label{prop:algBE_envalg}
    Let $\mathcal{C}$ be a symmetric monoidal category over $Ch(\mathbb{F})$. If an object $A$ in $\mathcal{C}$ is an $\mathcal{E}_+$-algebra then
    \begin{enumerate}[label=\roman*)]         
        \item we have a morphism $\env(\overline{A})\twoheadrightarrow A$ of $\env(\overline{A})$-modules,
        \item there is a morphism $A\otimes A^{op}\to \env(\overline{A})$,
        \item and there exists an isomorphism $\env(\overline{A}) \simeq \env(\overline{A})^{op}$
    \end{enumerate}
    where the last two morphisms are morphisms of $As$-algebras in $\mathcal{C}$. Furthermore, if $\mathcal{C}=Ch(\mathbb{F})$ or $\mathcal{C}=Ch(\mathbb{F})^{\pGM}$, the morphism of i) is a quasi-isomorphism. 
    \end{proposition}
    Before giving the proof, let's quickly mention a corollary. Let $\mathcal{X}$ be a topological space, we set $E(\mathcal{X}):=\env(\overline{C^\ast(\mathcal{X}; \mathbb{F})})$. This defines a contravariant functor from topological spaces to the category of DGAs. The functoriality of $E$ is deduced from the functoriality of the singular cochain complex and the enveloping algebra. By applying the previous proposition, we get the following result. 
    \begin{corollaire}\label{coro:algBE_cochain_envalg}
    Let $\mathbb{F}$ be a field. There exists a functor $E: Top^{op} \to DGA$ such that for any topological space $\mathcal{X}$ the following properties are verified:
    \begin{enumerate}[label=\roman*)]
        \item we have a quasi-isomorphism of left $E(\mathcal{X})$-modules $E(\mathcal{X})\xtwoheadrightarrow{\simeq} C^\ast(\mathcal{X}; \mathbb{F})$,            
        \item there is a DGA morphism $C^\ast(\mathcal{X}; \mathbb{F})\otimes C^\ast(\mathcal{X}; \mathbb{F})^{op}\to E(\mathcal{X})$,
        \item and there exists an isomorphism of DGA $E(\mathcal{X}) \simeq E(\mathcal{X})^{op}$.
    \end{enumerate}
    Furthermore, these morphisms induce natural transformations. 
    \end{corollaire}

    \begin{proof}[Proof of Proposition \ref*{prop:algBE_envalg}]
        The map of i) is given by
        \[\begin{array}{lccc}
        f:&\env(\overline{A}) & \to & A \\
         ~&\sigma(t, a_1, \ldots, a_k)& \mapsto & \sigma(1_A, a_1, \ldots, a_k). 
        \end{array}\]
        It is a morphism of left $\env(\overline{A})$-modules by corollary \ref{coro:morph_env_alg}. We can see that $f$ admits a right inverse 
        \[\begin{array}{lccc}
        g:& A & \to &\mathcal{U}_{\mathcal{E}}(\overline{A}) \\
         ~&a& \mapsto & (1,2)(t,a). 
        \end{array}\]
        We're going to show that $g \circ f$ and $\id_{\env(\overline{A})}$ are chain homotopic when $\mathcal{C}=Ch(R)$. We first study the case of free $\mathcal{E}_+$-algebras. We suppose that $A=\mathcal{E}_+(V):=S(\mathcal{E}_+, V)$ for some chain complex $V$ and with $\mathcal{E}_+(V):=\oplus_{k\geq 0}(\mathcal{E}_+(k)\otimes V^{\otimes k})_{\mathfrak{S}_k}=\mathcal{E}(0)\oplus\mathcal{E}(V)$. Note that $1_A$, the unit of $A$, is identified with an element in $\mathcal{E}(0)$. Furthermore, we have $\overline{A}=\mathcal{E}(V)$ and $\env(\overline{A})=S(\mathcal{E}[1], V)$. If $\sigma\in \mathcal{E}(k+1)$, we have $f(\sigma(t, a_1, \ldots, a_k))=F_k(\sigma)(a_1, \ldots, a_k)$ with $F_k$ being the composite 
        \[\mathcal{E}(k+1)\hookrightarrow \mathcal{E}_+(k+1)\otimes \mathcal{E}_+(0)\otimes \mathcal{E}_+(1)\otimes \ldots \otimes \mathcal{E}_+(1) \xrightarrow{\gamma} \mathcal{E}_+(k)=\mathcal{E}(k)\]
        where $\gamma$ refers to the composition product in $\mathcal{E}_+$. Note that $F$ is $\mathfrak{S}_k$-equivariant. 
        
        Let $a=\sigma(v_1, \ldots, v_k)\in A$, note that we actually mean the class of elements which are identified with $\sigma(v_1, \ldots, v_k)$ under the action of $\mathfrak{S}_k$. We have $g(a)=(1, 2)(t, \sigma(v_1, \ldots, v_k))=((1, 2)\circ_2 \sigma )(t, v_1, \ldots v_k)$ using the equivariance relations in the envelopping algebra. Hence, $g$ is actually induced by the composition products in $\mathcal{E}$, we have $g(\sigma(v_1, \ldots, v_k))=G_k(\sigma)(t, v_1, \ldots, v_k)$ where 
        \[
        \begin{array}{lccc}
             G_k:& \mathcal{E}(k) &\to& \mathcal{E}(k+1)  \\
             ~&  \sigma &\mapsto& (1, 2)\circ_2 \sigma.
        \end{array}
        \]
        This map is also $\mathfrak{S}_k$-equivariant. We have $F_k\circ G_k=\id_{\mathcal{E}(k)}$ and $F_k$ and $G_k$ are chain maps since they are induced by composition products. The complex $\mathcal{E}(k)$ has trivial homology modules since it is the bar construction of $\mathfrak{S}(k)$ (we can present a contracting homotopy). We deduce that $F_k$ and $G_k$ are quasi-isomorphisms between degree-wise free chain complexes and hence, $G_k\circ F_k$ is chain homotopic to $\id_{\mathcal{E}(k+1)}$. We have a chain homotopy $H:\mathcal{E}(k+1)\to \mathcal{E}(k+1)$ (of degree +1) such that
        \[d_{\mathcal{E}(k+1)}\circ H + H\circ d_{\mathcal{E}(k+1)} = G_k\circ F_k - \id_{\mathcal{E}(k+1)}.\]
         Since $G_k\circ F_k$ is $\mathfrak{S}_k$-equivariant, we can lift this $H$ into a map $h:\env(\overline{A})\to \env(\overline{A})$ where $h(\sigma(t, a_1, \ldots, a_k)):=H(\sigma)(t, a_1, \ldots, a_k)$ for $\sigma(t, a_1, \ldots, a_k)\in\env(\overline{A})$. From the above given equality, we get that
         \begin{multline}\label{eq1}
             d_{\mathcal{E}(k+1)}\circ H(\sigma)(t, a_1,\ldots, a_k) + H\circ d_{\mathcal{E}(k+1)}(\sigma)(t, a_1,\ldots, a_k) \\ = G_k\circ F_k(\sigma)(t, a_1,\ldots, a_k) - \sigma(t, a_1,\ldots, a_k).
         \end{multline}
        The differential of $\env(\overline{A})$ is given by the differential of a tensor product of chain complexes. For $\sigma(t, a_1, \ldots, a_k)\in\env(\overline{A})$, we have
        \[d_{\env(\overline{A})}(\sigma(t, a_1, \ldots, a_k))= d_{\mathcal{E}(k+1)}(\sigma)(t, a_1, \ldots, a_k)+ (-1)^{\Real{\sigma}}\sigma\left(d_{A^{\otimes k}}(t, a_1,\ldots, a_k)\right)\]
        where $d_{A^{\otimes k}}(t, a_1,\ldots, a_k)=\sum_{i=1}^k (-1)^{\eps_i} (t, a_1, \ldots, d_A(a_i),\ldots, a_k)$ with $\eps_i=\sum_{j=1}^i \Real{a_j}$. Note that we have
        \begin{multline}\label{eq2}
            d_{\env(\overline{A})}(H(\sigma)(t, a_1, \ldots, a_k))+h(d_{\env(\overline{A})}(\sigma(t, a_1, \ldots, a_k)))\\= d_{\mathcal{E}(k+1)}\circ H(\sigma)(t, a_1, \ldots, a_k)+h(d_{\mathcal{E}(k+1)}(\sigma)(t, a_1, \ldots, a_k))
        \end{multline}
        since
        \[(-1)^{\Real{H}+\Real{\sigma}}H(\sigma)\left(d_{A^{\otimes k}}(t, a_1,\ldots, a_k)\right)=(-1)^{1+\Real{\sigma}}h\left(\sigma(d_{A^{\otimes k}}(t, a_1,\ldots, a_k))\right).\]
        Finally, using (\ref{eq1}) and (\ref{eq2}), we get that $d_{\env(\overline{A})}\circ h + h\circ d_{\env(\overline{A})} = g\circ f - \id_{{\env(\overline{A})}}$. This proves that $f$ and $g$ are chain homotopy equivalences, whence they are quasi-isomorphisms. We now explain why this results holds for any $\mathcal{E}_+$-algebra (see \cite[Chapter 4]{Fre09}). Note that an $\mathcal{E}_+$-algebra $A$ can be obtained as the coequalizer of free $\mathcal{E}_+$-algebras. We have the following commutative diagram
        \[\begin{tikzcd}
        	{\mathcal{E}_+(S(\mathcal{E}_+,A))} & {\mathcal{E}_+(A)} & A \\
        	{S(S(\mathcal{E}[1],A),A)} & {S(\mathcal{E}[1],A)} & {\mathcal{U}_{\mathcal{E}}(\overline{A})}
        	\arrow[shift left, from=1-1, to=1-2]
        	\arrow[shift right, from=1-1, to=1-2]
        	\arrow[from=1-2, to=1-3]
        	\arrow[shift left, from=2-1, to=2-2]
        	\arrow[shift right, from=2-1, to=2-2]
        	\arrow[from=2-2, to=2-3]
        	\arrow["f"', shift right, dashed, from=2-3, to=1-3]
        	\arrow["{f_1}"', shift right, from=2-2, to=1-2]
        	\arrow["{f_2}"', shift right, from=2-1, to=1-1]
        	\arrow["{g_2}"', shift right, from=1-1, to=2-1]
        	\arrow["{g_1}"', shift right, from=1-2, to=2-2]
        	\arrow["g"', shift right, dashed, from=1-3, to=2-3]
        \end{tikzcd}\]
        where for $i=1,2$, $f_i$ and $g_i$ are chain homotopy equivalences. Thinking of $A$ and $\env(\overline{A})$ as colimits, we get morphisms $f$ and $g$ and we can check that they are also chain homotopy equivalences. We'll deal with the case $\mathcal{C}=Ch(\mathbb{F})^{\pGM}$ in subsection \ref{subsec:thm_BV_Hoch_bup}.
        
        For ii), notice that for any $k\in\mathbb{N}$, $As(k)$ is isomorphic to the degree 0 component of $\mathcal{E}(k)$. This induces a morphism of operads $As\to \mathcal{E}$. Thus, by proposition \ref{prop:functo_envelop}, we have a morphism $A\otimes A^{op} \simeq \mathcal{U}_{As}(\overline{A})\to \env(\overline{A})$. 
        
        The last point is proved in the next subsection. 
    \end{proof}
    \begin{rem}
        In \cite[Appendix A]{FG20}, Fresse and Guerra proved that we have a quasi-isomorphism $A\vee B\to A\otimes B$ between the coproduct and the tensor product of two $\mathcal{E}_+$-algebras in $Ch(\mathbb{F})$. Point i) can also be deduced from their result. 
    \end{rem}
        \subsubsection{The mirror morphism}    
    Let $(\mathcal{C}, \otimes, \mathcal{I})$ be a symmetric monoidal category over $Ch(\mathbb{F})$. We will denote by $As$ the associative operad in $Ch(\mathbb{F})$. Notice that if $A$, an object in $\mathcal{C}$, is an algebra over $As$ then, we have isomorphisms of algebras (monoids in $\mathcal{C}$)
    \[\mathcal{U}_{As}(A) \simeq A_+\otimes A_+^{op} \simeq  A_+^{op}\otimes A_+ \simeq \mathcal{U}_{As}(A)^{op}\]
    where $A_+\simeq A \oplus \mathcal{I}$.
    This implies that we have an equivalence between the category of left $\mathcal{U}_{As}(A)$-modules and the category of right $\mathcal{U}_{As}(A)$-modules. In an unpublished paper, David Chataur notices that we also have a similar morphism for the envelopping algebra over $\mathcal{E}$, \emph{the mirror morphism}, which he denotes by $\mathfrak{m}$. Furthermore, it fits in the following commutative diagram of algebras when $A$ is an $\mathcal{E}$-algebra.
    \[
    \begin{tikzcd}
    \mathcal{U}_{\mathcal{E}}(A) \arrow[r, "\mathfrak{m}"] & \mathcal{U}_{\mathcal{E}}(A)^{op}   \\
    \mathcal{U}_{As}(A) \arrow[r, "T_{op}"] \arrow[u]     & \mathcal{U}_{As}(A)^{op}. \arrow[u]
    \end{tikzcd}
    \]
    We first describe the isomorphism $T_{op}$.
    \begin{proposition}
       Let $(\mathcal{C}, \otimes, \mathcal{I})$ be a symmetric monoidal category over $Ch(\mathbb{F})$ and let $A$ be an object in $\mathcal{C}$ endowed with and algebra structure over $As$. We have a commutative diagram of algebra isomorphisms
        \[
        \begin{tikzcd}
        A_+\otimes A_+^{op} \arrow[r, "T"] \arrow[d, "\simeq"] & A_+^{op}\otimes A_+  \arrow[d, "\simeq"] \\
        \mathcal{U}_{As}(A) \arrow[r, "T_{op}"]      & \mathcal{U}_{As}(A)^{op}
        \end{tikzcd}
        \]        
        where $T(a\otimes b)=(-1)^{\Real{a}\Real{b}}b\otimes a$ for $a\otimes b\in A_+\otimes A_+^{op}$. 
    \end{proposition}
    \begin{proof}
        The isomorphism of unital algebras $\phi: A_+\otimes A_+^{op}\to \mathcal{U}_{As}(A)$ maps the unit $1_A\otimes 1_A$ to $1_1(t)$. An element $a\otimes b\in A\otimes A^{op}$ is mapped to $(2, 1, 3)(t, a, b)$. Conversely, $\phi^{-1}$ sends an element $\sigma(a_1, \ldots a_k)\in \mathcal{U}_{As}(A)$ (with $\sigma\in\mathfrak{S}_k$, $a_i\in A$ for $i\in\{2, \ldots k\}$ and $a_1$ a variable) to $a_{\sigma(1)}.a_{\sigma(2)}\ldots a_{\sigma(i-1)}\otimes a_{\sigma(i+1)}\ldots a_{\sigma(k)}$ where $i=\sigma^{-1}(1)$. The composite $\phi\circ T\circ \phi^{-1}$ sends $\overline{a}_\sigma=\sigma(a_1, \ldots a_n)\in \mathcal{U}_{As}(A)$ to $(2, 1, 3)(-1)^{\eps_i\eta_i}(a_1, a_{\sigma(i+1)}\ldots a_{\sigma(k)}, a_{\sigma(1)}.a_{\sigma(2)}\ldots a_{\sigma(i-1)})$ where $\eps_i=\sum_{j=1}^{i-1}\Real{a_{\sigma(j)}}$ and $\eta_i=\sum_{j=i+1}^{n}\Real{a_{\sigma(j)}}$. Note that the product $a_1.a_2\ldots a_j$ for $j\in\{1,\ldots, k\}$ corresponds to $(1,2,\ldots, j)(a_1, a_2,\ldots, a_j)$, this implies that 
        \[(2, 1, 3)(a_1, a_{\sigma(i+1)}\ldots a_{\sigma(k)}, a_{\sigma(1)}\ldots a_{\sigma(i-1)})=\qquad\] \[(2, 1, 3)(a_1, (1, \ldots, k-i)(a_{\sigma(i+1)},\ldots, a_{\sigma(k)}), (1, \ldots, i-1)(a_{\sigma(1)},\ldots, a_{\sigma(i-1)})).\]
        Using the composition product found in the permutation operad, this term corresponds to 
        \[(2, \ldots, k-i+1, 1, k-i+2,\ldots, k)(a_1, a_{\sigma(i+1)},\ldots, a_{\sigma(k)}, a_{\sigma(1)},\ldots, a_{\sigma(i-1)}).\]
        We denote this element by $\overline{a}^{op}_\sigma$. We have $\overline{a}_\sigma\in \mathcal{U}_{As}(A)(k-1)$, it's an element which is invariant under the action of $\mathfrak{S}_{k-1}$ or equivalently the action of elements $\omega \in \mathfrak{S}_k$ such that $\omega(1)=1$. We consider $\omega=(1, \sigma(i+1), \ldots, \sigma(k), \sigma(1), \ldots, \sigma(i-1))$. The action of $\omega$ on $\overline{a}^{op}_\sigma$ gives the element $\omega.\overline{a}^{op}_\sigma$ which is equal to
        \[(\sigma(i+1), \ldots, \sigma(k), 1, \sigma(1), \ldots, \sigma(i-1))(a_1,\ldots, a_k).\]
        Indeed when $\omega$ permutes the element of $\overline{a}^{op}_\sigma$, it sends an element $a_{\sigma(j)}$ to the $\sigma(j)^{\text{th}}$-position. Finally we define $T_{op}$ for $\sigma(a_1, \ldots a_k)\in \mathcal{U}_{As}(A)$ by setting 
        \[T_{op}(\sigma(a_1, \ldots a_k)):=(-1)^{\eps_i\eta_i}\sigma^{op}(a_1, \ldots a_k)\] 
        where $\sigma^{op}=(\sigma(i+1), \ldots, \sigma(k), 1, \sigma(1), \ldots, \sigma(i-1))$, $i=\sigma^{-1}(1)$, $\eps_i=\sum_{j=1}^{i-1}\Real{a_{\sigma(j)}}$ and $\eta_i=\sum_{j=i+1}^{k}\Real{a_{\sigma(j)}}$. 
    \end{proof}
    In the previous proof, we've defined an operation on the permutation operad: 
    \begin{align*}
        \mathfrak{S} &\to \mathfrak{S} \\
        \sigma &\mapsto \sigma^{op}.
    \end{align*}
    We will use this operation to define an isomorphism of algebras between $\env(A)$ and $\env(A)^{op}$. To do so, we will need the following result.
    \begin{lemma}
            For $k\in\mathbb{N}^\ast$ and for any $\sigma, \omega \in \mathfrak{S}(k)$, we have the following relation
            \[(\sigma\circ_1 \omega)^{op}=\omega^{op}\circ_1\sigma^{op}.\]
    \end{lemma}
    \begin{proof}
        Let $k\in\mathbb{N}^\ast$, we consider to permutations $\sigma, \omega \in \mathfrak{S}(k)$ with $\sigma=(\sigma(1), \ldots, \sigma(k))$ and $\omega=(\omega(1), \ldots, \omega(k))$. Wet set $i=\sigma^{-1}(1)$ and $j=\omega^{-1}(1)$. On one side, we have
        \begin{multline*}
        \sigma\circ_1\omega=\sigma(\omega , 1_1, \ldots 1_1)=(\sigma(1)+k-1, \ldots, \sigma(i-1)+k-1, \\
        \omega(1), \ldots,\omega(j-1), 1, \omega(j+1), \ldots, \omega(k),\sigma(i+1)+k-1, \ldots, \sigma(k)+k-1).
        \end{multline*}
        Hence, the term is equal to
        \begin{multline*}
        (\sigma\circ_1\omega)^{op}=(\omega(j+1), \ldots, \omega(k),\sigma(i+1)+k-1, \ldots, \sigma(k)+k-1, 1, \\
         \sigma(1)+k-1, \ldots, \sigma(i-1)+k-1, \omega(1), \ldots,\omega(j-1)).  
        \end{multline*}
        On the other side, since $\sigma^{op}=(\sigma(i+1), \ldots, \sigma(k), 1, \sigma(1),\ldots \sigma(i-1))$ and $\omega^{op}=(\omega(j+1), \ldots, \omega(k+1), 1, \omega(1),\ldots \omega(j-1))$, we get that $\omega^{op}\circ_1\sigma^{op}$ is equal to
        \begin{multline*}
            \omega^{op}\circ_1\sigma^{op}=(\omega(j+1)+k-1, \ldots, \omega(k)+k-1,\sigma(i+1), \ldots, \sigma(k), 1,\\ 
            \sigma(1), \ldots, \sigma(i-1), \omega(1)+k-1, \ldots,\omega(j-1)+k-1).
        \end{multline*}
        Using the equivariance relation in the permutation operad, we can see that $(\sigma\circ_1\omega)^{op}$ and $\omega^{op}\circ_1\sigma^{op}$ are identified in $\mathfrak{S}$.
    \end{proof}
    We can now prove the result we mentionned at the beginning of this subsection.
    \begin{proposition}\label{prop:env_alg_BE_equiv_module_cats}
     Let $(\mathcal{C}, \otimes, \mathcal{I})$ be a symmetric monoidal category over $Ch(\mathbb{F})$ and let $A$ be an object in $\mathcal{C}$. If $A$ is an $\mathcal{E}$-algebra then, we have a commutative diagram of isomorphisms of algebras
    \[
    \begin{tikzcd}
    \mathcal{U}_{\mathcal{E}}(A) \arrow[r, "\mathfrak{m}"] & \mathcal{U}_{\mathcal{E}}(A)^{op}   \\
    \mathcal{U}_{As}(A) \arrow[r, "T_{op}"] \arrow[u]     & \mathcal{U}_{As}(A)^{op}. \arrow[u]
    \end{tikzcd}
    \]
    \end{proposition}    
    \begin{proof}
        We can extend the definition of the operator $op$ to the the Barratt-Eccles operad. For $k\in\mathbb{N}$, we consider the map $op$ defined for $l\in \mathbb{N}$ by
        \[\begin{array}{rccc}
            op: & \mathcal{E}(k)_l & \to & \mathcal{E}(k)_l  \\
             & \sigma_0\otimes \ldots \otimes \sigma_l & \mapsto & \sigma_0^{op}\otimes \ldots \otimes \sigma_l^{op}.
        \end{array}\]
        This gives a chain map on $\mathcal{E}(k)$. Let $\sigma=\sigma_0\otimes \ldots \otimes \sigma_l\in\mathcal{E}(k)_l$. Recall that the differential $\mathcal{E}(k)$ on is given by
        \[d(\sigma_0\otimes \ldots\otimes \sigma_l)=\sum_{i=0}^l (-1)^i \sigma_0\otimes \ldots\otimes \sigma_{i-1} \otimes \sigma_{i+1}\ldots\otimes \sigma_l.\]
        We have $d(op(\sigma)= \sum_{i=0}^l (-1)^i\sigma_0^{op}\otimes \ldots \sigma_{i-1}^{op}\otimes \sigma_{i+1}^{op} \otimes \sigma_l^{op}=d(\sigma)^{op}$.
        We can show that for any $\sigma, \omega \in \mathcal{E}$, we have
        \[op(\omega)\circ_1 op(\sigma)=op(\sigma \circ_1 \omega).\]
        Indeed, recall that in the Barratt-Eccles operad, the $1$-partial composition product of $\sigma=(\sigma_0, \ldots, \sigma_d)\in\mathcal{E}(k)_d$ and $\omega=(\omega_0, \ldots, \omega_e)\in\mathcal{E}(l)_e$ is defined as follows
        \[\sigma\circ_1\omega:=\sum_{(x_\ast, y_\ast)}\pm (\sigma_{x_0}\circ_1 \omega_{y_0}, \ldots, \sigma_{x_{d+e}}\circ_1 \omega_{y_{d+e}})\]
        where the sum ranges over the set of paths from $(0,0)$ to $(d,e)$ in a $d\times e$ diagram. By the previous lemma, we have 
        \[op(\sigma\circ_1\omega)=\sum_{(x_\ast, y_\ast)}\pm (\omega_{y_0}^{op}\circ_1 \sigma_{x_0}^{op}, \ldots, \omega_{y_{d+e}}^{op}\circ_1 \sigma_{x_{d+e}}^{op}).\]
        On the other hand,
        \[op(\omega)\circ_1 op(\sigma)=\sum_{(a_\ast, b_\ast)}\pm (\omega_{a_0}^{op}\circ_1 \sigma_{b_0}^{op}, \ldots, \omega_{a_{e+d}}^{op}\circ_1 \sigma_{b_{e+d}}^{op})\]
        where the sum ranges over the set of paths from $(0,0)$ to $(e,d)$ in a $e\times d$ diagram which is in bijection with paths in a $d\times e$ diagram. Hence, we get the sought equality.
        We define $\mathfrak{m}:\env(A)\to \env(A)^{op}$ by setting for $\omega(t, a_1, \ldots, a_k)\in \env(A)$
        \[\mathfrak{m}(\omega(t, a_1, \ldots, a_k)):=(-1)^{\eps_i \eta_i}op(\omega)(t, a_1, \ldots, a_k))\]
         where $i=\sigma^{-1}(1)$, $\eps_i=\sum_{j=1}^{i-1}\Real{a_{\sigma(j)}}$ and $\eta_i=\sum_{j=i+1}^{k}\Real{a_{\sigma(j)}}$. This is a morphism of algebras because of the identity verified by the $op$ operator. One easily checks that it makes the diagram given in the proposition commute. 
    \end{proof}
    
    \subsection{Proof of Theorem \ref{thm:BV-Hoch-via-operads}}\label{subsec:Thm_BV_sing_proof}
        In this section, we explain why the following statement is true.
        \begin{thm*}[\textbf{A}]
            Let $\mathcal{M}$ be a compact, oriented, smooth manifold. Then, there exists a Batalin-Vilkovisky algebra structure on $HH^\ast(C^\ast(\mathcal{M}; \mathbb{F}))$, the Hochschild cohomology of the singular cochains of $\mathcal{M}$ with coefficients in a field $\mathbb{F}$.
        \end{thm*}
        We actually show a more general result. We get Theorem \ref{thm:BV-Hoch-via-operads} as a consequence of Poincaré duality, the following proposition and using the fact that the singular cochains of a topological space is endowed with the structure of an $\mathcal{E}_+$-algebra (Corollary \ref{coro:BE alg on norm cochains}). 
        \begin{proposition}\label{prop:BV_HH_Alg_BE}
            Let $A$ be an algebra over the Barratt-Eccles operad $\mathcal{E}_+$ in $Ch(\mathbb{F})$. If there exists a quasi-isomorphism of $A$-modules between $A$ and its linear dual $DA$ then one can endow the Hochschild cohomology $HH^\ast(A)$ with a Batalin-Vilkovisky algebra structure. 
        \end{proposition}
        We will need some preliminary results in order to prove this proposition. 

        Let $A$ be a DGA. Assuming there exists a quasi-isomorphism of $A$-modules between $A$ and its linear dual $DA$, Menichi gives in \cite{Men09} sufficient conditions to endow the Hochschild cohomology $HH^\ast(A)$ with a BV-algebra structure. We recall these results below.
        \begin{defi}[{\cite[Definition 4.4]{Abb15}}]\label{def:DPDA-Abb}
            A differential graded algebra $A$ is a \emph{derived Poincaré duality algebra} (DPDA) if $A$ is isomorphic to its linear dual $DA$ (eventually up to a shift) in the derived category of $A$-bimodules. 
        \end{defi}    
        Let $ev(1_A):\Ext_A(A, DA)^\ast\to H^\ast(DA)$ be the evaluation at the unit $1_A$ of $A$ and let $i_A:A\hookrightarrow A\otimes A^{op}$ be the inclusion in the first factor. We set $\eval:=ev(1_A)\circ Ext_{i_A}(A, DA)$. The following result tells us that if there exists a section to $\Ext_{i_A}(A, DA)$, then that morphism preserves quasi-isomorphisms. 
        \[\begin{tikzcd}
    	{H^\ast(DA)} & {\Ext_{A}(A, DA)^\ast} & {HH^\ast(A, DA):=\Ext_{A^e}(A, DA)^\ast}
    	\arrow["{ev(1_A)}"', "\simeq", from=1-2, to=1-1]
    	\arrow["{\Ext_{i_A}(A, DA)}", curve={height=-18pt}, from=1-3, to=1-2]
    	\arrow["{?}", curve={height=-18pt}, dashed, from=1-2, to=1-3]
        \end{tikzcd}\]
        
        \begin{proposition}[{\cite[Proposition 11]{Men09}}]\label{prop:11_Menichi}
        Let $A$ be a DGA and let $c\in HC^\ast(A,DA)$ be a cocycle such that
        the morphism of $H(A)$-modules
        \begin{align*}
            H(A) &\xrightarrow{\simeq} H(DA) \\
            a &\mapsto a.\eval([c])    
        \end{align*}
        is an isomorphism.
        Then $c$ is a quasi-isomorphism and hence, $A$ is a derived Poincaré dualité algebra.
        \end{proposition}
        The next result gives an additional condition (a sort of cyclicity condition) that needs to be verified by a DPDA to endow $HH^\ast(A)$ with a BV algebra structure.
  
        \begin{proposition}[{\cite[Proposition 12]{Men09}}]\label{prop:12_Menichi}
        Let $A$ be a derived Poincaré duality algebra and let $\phi:HH^\ast(A)\to HH^\ast(A,DA)$ be the associated isomorphism of $HH^\ast(A)$-bimodules. We define a map $\Delta: HH^\ast(A) \to HH^{\ast-1}(A)$ by setting 
        \[\Delta=\phi^{-1}\circ B^\vee \circ \phi\]
        where $B^\vee$ is the \emph{dual of Connes boundary map}. If $\Delta(1)=0$, then the Gerstenhaber algebra $HH^\ast(A)$ equipped with $\Delta$ becomes a BV-algebra.
        \end{proposition}        

        We now proceed to the proof of Proposition \ref{prop:BV_HH_Alg_BE}.
        \begin{proof}
        Let $A$ be an $\mathcal{E}_+$-algebra equipped with a quasi-isomorphism of $A$-modules $Dual_A: A\to DA$. We set $\Gamma:=\Dual_A(1_A)$ where $1_A$ denotes the unit of the DGA $A$. We divide this proof into two parts:
        \begin{itemize}
            \item we first show that $A$ is a DPDA,
            \item then, we prove that $HH^\ast(A)$ is BV-algebra.
        \end{itemize}
        \indent \textbf{$A$ is a derived Poincaré duality algebra.}
        We're going to describe morphisms of $\env(\overline{A})$-modules
        \begin{equation}\label{diag:qiso_env_alg}
            A \xtwoheadleftarrow[f]{\simeq} \env(\overline{A}) \to DA.
        \end{equation} 
        The quasi-isomorphism on the left is given by i) of Proposition \ref{prop:algBE_envalg}.  The morphism of left $\env(\overline{A})$-modules $f:\env(\overline{A})\to A$ is characterized by $f(1_1(t))=1_{A}$ where $1_{A}$ denotes the unit of the algebra $A$. Note that this morphism is surjective, for $a\in A$, we have $f((2, 1)(t, a))=(2, 1)(1_{A}, a)=a$.
        
        Let's now describe the left $\env(\overline{A})$-module structure found on $DA$. By iii) of Proposition \ref{prop:algBE_envalg}, there exists a right $\env(\overline{A})$-module structure on $A$. For $\sigma(t, a_1, \ldots, a_n)\in \env(\overline{A})$ and $a\in A$, it is given by 
        \[a.\sigma(t, a_1, \ldots, a_n):=\sigma^{op}(t, a_1, \ldots, a_n).a=\sigma^{op}(a, a_1, \ldots, a_n)\]
        Hence there exists a left $\env(\overline{A})$-module structure on $DA$. For $\sigma(t, a_1, \ldots, a_n)\in \env(\overline{A})$ and $h\in DA$, we denote this action by $\sigma(t, a_1, \ldots, a_n).h$. It is the linear form in $DA$ that sends an element $a\in A$ to $h(\sigma^{op}(t, a_1, \ldots, a_n).a)=h(\sigma^{op}(a, a_1, \ldots, a_n))$. We define morphism of left $\env(\overline{A})$-modules $g: \env(\overline{A})\to A$ by setting $g(1_{\env(\overline{A})})=\Gamma$. 
        
        By ii) of Proposition \ref{prop:algBE_envalg}, the maps given above are actually morphisms of $A$-bimodules.
        Let $q_{A}:P\to A$ be a cofibrant approximation of $A$ in the category of $A$-bimodules, there exists a morphism of $A$-bimodules between $P$ and $\env(\overline{A})$ which makes the diagram below commute
        \[\begin{tikzcd}
    	                   &  {P} \\
    	{A}    & {\env(\overline{A})}.
    	\arrow[two heads, "\simeq", from=2-2, to=2-1]
    	\arrow["\simeq"', from=1-2, to=2-1]
    	\arrow[dashed, from=1-2, to=2-2]
        \end{tikzcd}\] 
        Finally, we get the following commutative diagram of $A$-bimodules
        \[
            \begin{tikzcd}
            	& {P} \\
            	{A} & {\env(\overline{A})} & {DA}
            	\arrow["\simeq"', two heads, from=2-2, to=2-1]
            	\arrow[from=2-2, to=2-3]
            	\arrow["\simeq"', from=1-2, to=2-1]
            	\arrow["\simeq"', from=1-2, to=2-2]
            	\arrow[from=1-2, to=2-3]
            \end{tikzcd}
        \]
        In particular, we have a morphism of $A$-bimodules $c:P \to DA$. By construction of the diagram above, we have $\eval([c])=[\Gamma]$. Proposition \ref{prop:11_Menichi} implies that $A$ is a DPDA.

        \textbf{Batalin-Vilkovisky algebra structure on $HH^\ast(A)$.} To use Proposition \ref{prop:12_Menichi}, we need to check that $B^\vee(c)=0$, this will be easier to prove if we have an explicit description of $c$. A cofibrant approximation of $A$ is given by 
        \[q_{A}:\mathbb{B}(A)\xtwoheadrightarrow{\simeq} A \text{ defined by }
        q_{A}(a[a_1, \ldots, a_k]b)=\begin{cases}
            0 & \text{if } k >0 \\
            a.b & \text{if } k=0 
        \end{cases}\]
        with $\mathbb{B}(A)$ the normalized two-sided bar construction of $A$. We have the following commutative diagram
        \[
            \begin{tikzcd}
            	& {\mathbb{B}(A)} \\
            	{A} & {\env(\overline{A})} & {DA}.
            	\arrow["\overline{f}", two heads, from=2-2, to=2-1]
                \arrow[from=1-2, to=2-3]
            	\arrow["\overline{g}"', from=2-2, to=2-3]
            	\arrow["q_{A}"', two heads, from=1-2, to=2-1]
            	\arrow["q_{\mathcal{U}}", from=1-2, to=2-2]
            \end{tikzcd}
        \]
        The morphism $q_{\mathcal{U}}$ is unique up to homotopy. For $a[a_1, \ldots, a_k]b\in \mathbb{B}(A)$, we set
        \[q_{\mathcal{U}}(a[a_1, \ldots, a_k]b):=\begin{cases}
            0 & \text{if } k >0 \\
            (2, 1, 3)(t, a, b) & \text{if } k=0. 
        \end{cases}\]
        Hence, the morphism $c: \mathbb{B}(A)\to DA$ is given by
        \[c(a[a_1, \ldots, a_k]b):=\begin{cases}
        0 & \text{if } k >0 \\
        (\alpha\mapsto \Gamma_{A}((2, 1, 3)^{op}(\alpha, a, b)) & \text{if } k=0
        \end{cases}\]
        with $(2, 1, 3)^{op}(\alpha, a, b)=(3, 1, 2)(\alpha, a, b)=ab\alpha$.
        The isomorphism
        \[HC^\ast(A, DA) \simeq D(HC_\ast(A)) \]
        sends a morphism of $A$-bimodules
        \begin{align*}
            \phi:   &\mathbb{B}(A) \to DA \\
            ~       &a[x]b \mapsto \phi(a[x]b)
        \end{align*}
        to a morphism $\overline{\phi}\in D(HC_\ast(A))$ given by
        \begin{align*}
            \overline{\phi}:    &A\otimes T(s\overline{A}) \to R \\
                    ~           &a[x] \mapsto \phi(1[x]1)(a).
        \end{align*}   
        In particular, for $\phi=c$, we have 
        \[\overline{\phi}(a[x])=
        \begin{cases}
            x.\Gamma_{A}(a) & \text{if } x\in R \\
            0 & \text{else.} 
        \end{cases}\]
        Using the definition of the dual of Connes boundary map, one can see that $B^\vee(c)=0$. This proves that $HH^\ast(A)$ is a BV-algebra.
        \end{proof}

\section{Analogous result for the blown-up cochain complex}\label{sect:bup_cochain}
    We will now prove the following theorem by adapting results from the previous section to the context of perverse differential graded algebras (pDGAs).
    \begin{letterthm}\label{thm:BV-Hoch-bup}
            Let $X$ be a compact, oriented, second-countable pseudomanifold. Then, there exists a perverse Batalin-Vilkovisky algebra structure on $HH^\ast_\bullet(\widetilde N_\bullet^\ast(X; \mathbb{F}))$, the Hochschild cohomology of the blown-up cochain complex of a pseudomanifold $X$ with coefficients over any field $\mathbb{F}$. 
    \end{letterthm}
    We start this section by recalling the definition of Hochschild cohomology for pDGAs and by giving some of its properties. A more detailed presentation is given in \cite{Raz23}.
    \subsection{Hochschild (co)homology for pDGAs}\label{subsect:Hoch_via_bar}
    In this section, we fix $n\in\mathbb{N}$ and consider $P_n$, the poset of Goresky-MacPherson $n$-perversities. The reader will find details in appendix \ref{app:inter_hom}.
    \subsubsection{Definitions and algebraic stuctures}\label{subsub:Hoch_def_prop}
    Let $A_\bullet$ be a perverse differential (upper) graded algebra. It is unital, so it is equipped with a morphism $\eta:F_{\overline{0}}(\mathbb{S}^0)_\bullet\to A_\bullet$. We set $\overline{A}_\bullet:=A_\bullet/\eta(F_{\overline{0}}(\mathbb{S}^0))_\bullet$. In order to define the Hochschild (co)chain complex of a pDGA, we'll need the two-sided bar construction. 
    
    \begin{defi}
    The (normalized) \emph{two-sided bar complex} is the perverse cochain complex given by
    \[\mathbb{B}(A)_\bullet:=(A\boxtimes T(s\overline{A}) \boxtimes A)_\bullet\]
    where $T(s\overline{A})$ is the tensor algebra of $s\overline{A}_\bullet$ (Example \ref{ex:tensoralg}).
    
    More precisely, the component of perverse degree $\overline{r}\in P_X$ and degree $q$ is given by
    \[\mathbb{B}(A)^q_{\overline{r}}=\bigoplus_{k\in \mathbb{N}} \mathbb{B}_k(A)^q_{\overline{r}}\]
    where \[\mathbb{B}_k(A)^q_{\overline{r}}=(A \boxtimes (s\overline{A})^{\boxtimes k} \boxtimes A)^q_{\overline{r}}\] 
    \[=\colim_{\overline{r}_0+\ldots +\overline{r}_{k+1}\leq\overline{r}} \bigoplus_{i_0+\ldots+i_{k+1}=q}A^{i_0}_{\overline{r}_0}\otimes (s \overline{A})^{i_1}_{\overline{r}_1}\otimes \ldots \otimes  (s \overline{A})^{i_k}_{\overline{r}_k}\otimes A^{i_{k+1}}_{\overline{r}_{k+1}}.\] 
    The element $a\otimes s(a_1)\otimes s(a_2)\otimes\ldots\otimes s(a_k)\otimes b$ from $\mathbb{B}_k(A)^q_{\overline{r}}$ is denoted $a[a_1|a_2|\ldots|a_k]b$ and is said to be of \emph{length} $k$ and of \emph{degree} 
    \[q=\Real{a}+\Real{b}+ \sum_{i=1}^{k}\Real{s(a_i)}=\Real{a}+\Real{b}+ \sum_{i=1}^{k}\Real{a_i}-k.\]
    
    
    The differential on $\mathbb{B}(A)$, $D=d_0+d_1$, is defined by an internal (or vertical) differential
    \[d_0(a[a_1|a_2|\ldots|a_k]b)=d(a)[a_1|a_2|\ldots|a_k]b-\sum_{i=1}^{k}(-1)^{\eps_i}a[a_1|a_2|\ldots|d(a_i)|\ldots|a_k]b\] 
    \[+(-1)^{\eps_{k+1}}a[a_1|a_2|\ldots|a_k]d(b),\]
    and an external (or horizontal) differential
    \[d_1(a[a_1|\ldots|a_k]b)=(-1)^{\Real{a}}aa_1[a_2|\ldots|a_k]b+\sum_{i=2}^{k}(-1)^{\eps_i}a[a_1|a_2|\ldots|a_{i-1}a_i|\ldots|a_k]b\] 
    \[-(-1)^{\eps_{k+1}}a[a_1|a_2|\ldots|a_{k-1}]a_kb,\]
    where $\eps_i=\Real{a}+\sum_{j<i}\Real{s(a_j)}$.
    Notice that since we took the suspension of $A$, we have $\degr(D)=\degr(d_0+d_1)=1$. 
    \end{defi}

    \begin{defi}
        The (normalized) \emph{Hochschild chain complex} of $A_{\bullet}$ with coefficients in a pDG $A_{\bullet}$-bimodule $M_{\bullet}$ is \[HC^{\bullet}_\ast(A, M):=(M\boxtimes_{A^e}\mathbb{B}(A))^\ast_{\bullet}\simeq (M\boxtimes T(s\overline{A}))^\ast_{\bullet},\]
        equipped with a degree 1 differential $D_\ast=d_0+d_1$ which is defined by
    \[d_0(m[a_1|\ldots|a_k])=d_M(m)[a_1|a_2|\ldots|a_k]-\sum_{i=1}^{k}(-1)^{\eps_i}m[a_1|a_2|\ldots|d_A(a_i)|\ldots|a_k],\] 
    \[d_1(m[a_1|\ldots|a_k])=(-1)^{\Real{m}}ma_1[a_2|\ldots|a_k]+\sum_{i=2}^k(-1)^{\eps_i}m[a_1|a_2|\ldots|a_{i-1}a_i|\ldots|a_k]\] 
    \[-(-1)^{\eps_k\Real{s(a_k)}}a_km[a_1|a_2|\ldots|a_{k-1}],\]
    where $\eps_i=\Real{m}+\sum_{j< i}\Real{s(a_j)}.$ 
    Its homology $HH^{\bullet}(A, M)$ is the \emph{Hochschild homology} of $A_{\bullet}$ with coefficients in $M_{\bullet}$.
    \end{defi}
    
    \begin{defi}\label{def:Hochcocplx}
    The (normalized) \emph{Hochschild cochain complex} of $A_{\bullet}$ with coefficients in a pDG $A_{\bullet}$-bimodule $M_{\bullet}$ is \[HC^\ast_{\bullet}(A, M):=\Hom_{A^e}(\mathbb{B}(A), M)^\ast_{\bullet}\simeq \Hom_{Ch(R)^{P_X}}(T(s \overline{A}),M)^\ast_{\bullet}. \]
    Explicitly, the degree $q$ and perversity $\overline{r}$ component $HC^q(A, M)_{\overline{r}}$ is given by 
    \[\Hom_{Ch(R)^{P_X}}(T(s \overline{A}),M)^q_{\overline{r}}=\lim_{\overline{r}\leq \overline{q}-\overline{p}} \prod_{j-i=q} \Hom_R(T(s \overline{A})^i_{\overline{p}}, M^j_{\overline{q}}).\]
    A cochain $\phi\in HC_\bullet(A, M)$ will be identified with the unique element $f$ in the perverse chain complex $\Hom_{Ch(R)^{P_X}}(T(s \overline{A}),M)_\bullet$ which verifies 
    \[\phi(a_0[a_1|\ldots|a_k]a_{k+1})=a_0f([a_1|\ldots|a_k])a_{k+1}\]
    for any $a_0[a_1|a_2|\ldots|a_k]a_{k+1}$ in $\mathbb{B}_k(A)$.
    We denote the degree of a cochain $f\in HC^\ast_\bullet(A, M)$ by $\Real{f}$. The differential is given by $D^\ast=d_0+d_1$ where
    \[d_0(f)([a_1|\ldots| a_k])=d_Mf([a_1|\ldots|a_k])+\sum_{i=1}^{k}(-1)^{\eps_i+\Real{f}}f([a_1|\ldots|d_A(a_i)|\ldots|a_k]),\] 
    \[d_1(f)([a_1|\ldots|a_k])=-(-1)^{(\Real{a_1}+1)\Real{f}}a_1f([a_2|\ldots|a_k])+(-1)^{\eps_{k}+\Real{f}}f([a_1|\ldots|a_{k-1}])a_k\] \[-\sum_{i=2}^{k}(-1)^{\eps_i+\Real{f}}f([a_1|\ldots|a_{i-1}a_i|\ldots|a_k]),\]
    with $\eps_i=\sum_{j<i}\Real{s(a_j)}.$ Its homology $HH_\bullet(A, M)$ is the \emph{Hochschild cohomology} of $A$ with coefficients in $M$.
    \end{defi}
    \textbf{Notation} We will use the notations $HC_\ast^\bullet(A)$ and $HC^\ast_\bullet(A)$ to refer to the Hochschild chain and cochain complexes of $A_\bullet$ i.e. when $M_\bullet=A_\bullet$. 

    It is possible to define a \emph{perverse Gerstenhaber algebra} structure on the Hochschild cohomology of a pDGA, see \cite[Theorem A]{Raz23}. 
    \begin{thm*}\label{thm:Hoch_is_Gerst}
    The Hochschild cohomology $HH^\ast_\bullet(A)$ of a pDGA $A_\bullet$ is a \emph{perverse Gerstenhaber algebra}. More precisely, it is equipped with two products. For every $m,p\in \mathbb{Z}$ and $\overline{p},\overline{q}\in\pGM$ $(\text{verifying }\overline{p}+\overline{q}\leq \overline{t})$, we have 
    \begin{itemize}
    \item a \emph{cup product}: $-\cup -: HH^m_{\overline{p}}(A) \otimes HH^p_{\overline{q}}(A) \to HH^{m+p}_{\overline{p}\oplus\overline{q}}(A)$
    \item and a \emph{Gerstenhaber bracket}: $[-,-]: HH^m_{\overline{p}}(A) \otimes HH^p_{\overline{q}}(A) \to HH^{m+p-1}_{\overline{p}\oplus\overline{q}}(A)$
    \end{itemize}    
    such that 
    \begin{enumerate}[label=\alph*)]
        \item $\cup$ is an associative and graded commutative product,
        \item the suspended Hochschild cohomology $(sHH^\ast_\bullet(A), [-,-])$ is a Lie algebra. For $f,g,h\in HH^\ast_\bullet(A)$, the following identities hold
        \begin{itemize}
        \item graded skew-commutativity: $[f,g]=-(-1)^{(\Real{f}-1)(\Real{g}-1)}[g,f]$,
        \item Jacobi identity: $[[f,g],h]=[f,[g,h]]-(-1)^{(\Real{f}-1)(\Real{g}-1)}[g,[f,h]]$,
         \end{itemize}   
    \item and both operations satisfy the Leibniz rule: 
    \[[f, g\cup h]=[f,g]\cup h + (-1)^{(\Real{f}-1)\Real{g}}g\cup [f,h].\]
    \end{enumerate}
    \end{thm*}
    
    \textbf{Action on Hochschild cohomology}
    We recall the left action of $HH^\ast_\bullet(A)$ on $HH^\ast_\bullet(A,M)$ where $A_\bullet$ is a pDGA and $M_{\bullet}$ is a pDG $A_{\bullet}$-bimodule. First, notice that for any $\overline{p}, \overline{q}\in \pGM$, we have a map of cochain complexes of degree 0 and perverse degree $\overline{0}$
    \[-\boxtimes_A -:HC^\ast_{\overline{p}}(A)\otimes HC^\ast_{\overline{q}}(A,M)\to HC^\ast_{\overline{p}+\overline{q}}(A,A\boxtimes_A M).\]
    For $f\in HC^i_{\overline{p}}(A)$ and $g\in HC^j_{\overline{p}}(A,M)$, let $f\boxtimes_A g$ be the perverse cochain in $HC^{i+j}_{\overline{p}+\overline{q}}(A,A\boxtimes_A M)$ given by
    \[f\boxtimes_A g([a_1|\ldots |a_l])=\sum_{k=0}^l (-1)^{\Real{g}(\Real{a_1}+\ldots + \Real{a_k}-k)}f([a_1|\ldots|a_k])\boxtimes_A g([a_{k+1}|\ldots|a_l]).\]
    This map is natural with respect to $M$. Furthermore, we have an isomorphism of perverse $A_\bullet$-bimodules which is given by the action of $A_\bullet$ on $M_\bullet$:
    \begin{align*}
        (A\boxtimes_A M)_\bullet &\simeq M_\bullet \\
        a\boxtimes_A m &\mapsto a.m
    \end{align*}
    The map induced in homology by the composite of both maps gives the sought action.  

    \subsubsection{Model category structure on perverse chain complexes}
    It is possible to interpret Hochschild (co)homology in terms of derived functors. Hovey has shown that the (projective) model category structure found on $Ch(R)$ can be lifted to $Ch(R)^{\pGM}$. Just as we did in the previous subsection, we'll be working with perverse cochain complexes. 
    \begin{thm}[{\cite[Theorem 3.1]{Hov09}}]\label{thm:modelcat_pchain}
        There is a model category structure on the category of perverse chain complexes $Ch(R)^{\pGM}$ where a morphism $f:Z_\bullet\to Y_\bullet$ is weak equivalence or a fibration if and only if $f_{\overline{p}}:Z_{\overline{p}}\to Y_{\overline{p}}$ is so in $Ch(R)$ for all $\overline{p}\in \pGM$. 
    \end{thm}
    \begin{rem}\label{rem:quillen_adj_pch}
    This model structure is \emph{monoidal} \cite[Theorem 3.3]{Hov09}. The reader may refer to \cite[Chapter 4]{Hov99} for a detailed presentation of the notion. This implies, in particular, that we have the following \emph{Quillen adjunctions} in $Ch(R)^{\pGM}$ when $C_\bullet$ is cofibrant:
    \begin{itemize}
        \item[--] $(C\boxtimes -)_\bullet\vdash \Hom_{Ch(R)^{\pGM}}(C,-)_\bullet$,
        \item[--] $(-\boxtimes C)_\bullet\vdash \Hom_{Ch(R)^{\pGM}}(C,-)_\bullet$.
    \end{itemize}
    Furthermore, if $D_\bullet$ is fibrant $\Hom_{Ch(R)^{\pGM}}(-, D)$ preserves quasi-isomorphism between cofibrant objects. 
    \end{rem}
    The notion of chain homotopy can be adapted for perverse chain complexes.
    \begin{defi}
        A \emph{perverse chain homotopy} $h_\bullet:C_\bullet \to D_\bullet$ between two morphisms of perverse chain complexes $f_\bullet, g_\bullet: C_\bullet \to D_\bullet$ is given by a family of morphisms of $R$-modules $\{h_{\overline{p}}^k:C_{\overline{p}}^k\to D_{\overline{p}}^{k+1}\}_{k\in\mathbb{Z},\,\overline{p}\in\pGM}$ such that, for every $k\in\mathbb{Z}$ and $\overline{p},\overline{q}\in\pGM$ with $\overline{p}\leq \overline{q}$, we have 
        \[f^k_{\overline{p}}-g^k_{\overline{p}}=d_D\circ h^k_{\overline{p}}+h^{k+1}_{\overline{p}}\circ d_C\]
        and the following diagram commutes
        \[\begin{tikzcd}
        	{C_{\overline{p}}^k} & {C_{\overline{q}}^k} \\
        	{D_{\overline{p}}^{k+1}} & {D_{\overline{q}}^{k+1}}.
        	\arrow["{h_{\overline{p}}^k}"', from=1-1, to=2-1]
        	\arrow["{\phi_{\overline{p}\leq\overline{q}}}"', from=2-1, to=2-2]
        	\arrow["{h_{\overline{q}}^k}", from=1-2, to=2-2]
        	\arrow["{\phi_{\overline{p}\leq\overline{q}}}", from=1-1, to=1-2]
        \end{tikzcd}\]
        We then say that $f_\bullet$ and $g_\bullet$ are \emph{chain homotopy equivalences}.
    \end{defi}
    We can easily see that morphisms of perverse chain complexes which are chain homotopy equivalences are quasi-isomorphism. 
    
    Let $A_\bullet$ be a pDGA. We can endow the category of perverse DG $A_\bullet$-modules with a model category structure.

    \begin{thm}[{\cite[Theorem 3.4]{Hov09}}]
    Let $A_\bullet$ be a pDGA. There exists a model category structure on $pMod(A)$ where a morphism is a weak equivalence (or a fibration) if it is so in the underlying category $Ch(R)^{\pGM}$. 
    \end{thm}
    
    Derived functors are well defined now that we have a model category structure on bimodules over a pDGA.
    \begin{defi}
        Let $A_\bullet$ be a pDGA and $M_\bullet$ a perverse DG $A^e_\bullet$-module. For any $N_\bullet\in pMod(A^e)$ we define the Tor and Ext functors by
        \[\Tor_{A^e}(M,N)^\ast_\bullet:=H((P \boxtimes_{A^e} N)^\ast_\bullet)\]
        and
        \[\Ext_{A^e}(M,N)^\ast_\bullet:=H(\Hom_{A^e}(P,N)^\ast_\bullet)\]
        where $P_\bullet\to M_\bullet$ is a cofibrant approximation of the DG $A_\bullet^e$-module $M_\bullet$.
    \end{defi}    

    When we work over a field $\mathbb{F}$, we've shown in \cite[Section 2.4]{Raz23} that the two-sided bar construction $\mathbb{B}(A)_\bullet$ is a cofibrant approximation of the pDGA $A_\bullet$, this implies that
    \[HH^\bullet_\ast(A,M)=H_\ast(\mathbb{B}(A)\boxtimes_{A^e}M)_\bullet=\Tor^{A^e}_\ast(A,M)^\bullet\]
    and 
    \[HH^\ast_\bullet(A,M)=H^\ast(\Hom_{A^e}(\mathbb{B}(A), M)_\bullet)=\Ext^\ast_{A^e}(A,M)_\bullet.\]
        
    \subsection{Perverse $E_\infty$-algebras}
        The category of perverse chain complexes is a symmetric monoidal category over the category of chain complexes $Ch(R)$. Hence, by \cite[Section 3.2.1]{Fre09}, one can define $\mathcal{E}$-algebras in the category of perverse chain complexes. Note that the Barratt-Eccles operad is an example of $E_\infty$-operad (see Definition \ref{def:E_inf_op}) . In \cite[Definition 5.3]{CC22Einf}, Chataur and Cirici have studied $E_\infty$-algebras in the category of perverse chain complexes. We recall their definition. 
\begin{defi}
    A perverse cochain complex $A_\bullet$ is a \emph{perverse $\mathcal{E}_+$-algebra} if for any $m\in\mathbb{N}$ and for perversities $\overline{p}_1, \ldots, \overline{p}_m$, there exist evaluation morphisms
    \[\lambda: \mathcal{E}_+(m)\otimes (A_{\overline{p}_1}\otimes \ldots \otimes A_{\overline{p}_m})\to A_{\overline{p}_1\oplus\ldots\oplus \overline{p}_m}\]
    such that
    \begin{enumerate}
        \item the morphisms $\lambda$ satisfy natural unit, associativity and equivariance relations,
        \item for two $m$-tuples of perversities $\overline{p}_1, \ldots, \overline{p}_m$ and $\overline{q}_1, \ldots, \overline{q}_m$ verifying $\overline{p}_i\leq \overline{q}_i$ for all $1\leq i\leq m$, the following diagram commutes:
        \[
        \begin{tikzcd}
        \mathcal{E}_+(m)\otimes (A_{\overline{p}_1}\otimes \ldots \otimes A_{\overline{p}_m}) \arrow[r, "\lambda"] \arrow[d] & A_{\overline{p}_1\oplus\ldots\oplus \overline{p}_m} \arrow[d] \\
        \mathcal{E}_+(m)\otimes (A_{\overline{q}_1}\otimes \ldots \otimes A_{\overline{q}_m}) \arrow[r, "\lambda"]           & A_{\overline{q}_1\oplus\ldots\oplus \overline{q}_m}          
        \end{tikzcd}
        \]
        where the vertical maps are induced by the structure morphisms.
    \end{enumerate}
\end{defi}

    In \cite[Theorem A]{CST16Steenrod}, Chataur, Saralegui and Tanré have described an $\mathcal{E}_+(2)$-algebra structure on $\widetilde N(\underline{K}; \mathbb{F}_2)$ where $\underline{K}$ is a \emph{filtered face set}. They noticed that their proof could be generalized. Following a similar approach, we prove the statement given below. We decide to restrict ourselves to filtered spaces although the proof can be adapted to any filtered face set. 
    
    \begin{proposition}\label{prop:BE_alg_bup_cochains}
        Let $X$ be a stratified space of formal dimension $n$. The blown-up intersection cochain complex $\widetilde N^\ast_\bullet(X, R)$ is a perverse $\mathcal{E}_+$-algebra. Furthermore, this structure is functorial in $X$.
    \end{proposition}
    In the proof, we wil use the fact that the tensor product of two $\mathcal{E}_+$-algebras is also an $\mathcal{E}_+$-algebra. This was known by Berger and Fresse, it is deduced from the morphism of operads $\mathcal{E}_+\to\mathcal{E}_+\otimes \mathcal{E}_+$ they considered \cite{BF04}. 
    \begin{proof}
    We start by describing an $\mathcal{E}_+$-algebra structure on the blown-up cochain complex $\widetilde N^\ast(X, R)$. Recall that a cochain $c\in\widetilde N^\ast(X, R)$ associates to any regular simplex, $\sigma:\Delta=\Delta^{j_0}\ast\ldots\ast\Delta^{j_n}\to X$, an element $c_\sigma \in \widetilde N^\ast(\Delta)=N^\ast(c\Delta^{j_0})\otimes\ldots\otimes N^\ast(c\Delta^{j_{n-1}})\otimes N^\ast(\Delta^{j_n})$. We already know that there is an $\mathcal{E}_+$-algebra structure on $\widetilde N^\ast(\Delta)$:
    \[
    \begin{array}{rrcl}
    \widetilde\lambda:&\mathcal{E}_+(m)\otimes \widetilde N^\ast(\Delta)^{\otimes m}&\to& \widetilde N^\ast(\Delta)\\
    ~&\omega\otimes g_1\otimes\ldots\otimes g_m     &\mapsto&  \widetilde\lambda(\omega; g_1,\ldots,g_m)
    \end{array}
    \]
    We define evaluation morphisms on $\widetilde N^\ast(X, R)$
    \[\lambda:\mathcal{E}_+(m)\otimes \widetilde N^\ast(X, R)^{\otimes m}\to \widetilde N^\ast(X, R)\]
    by setting
    \[\lambda(\omega; c_1,\ldots,c_m)_\sigma:=\widetilde\lambda(\omega; c_{1,\sigma},\ldots,c_{m,\sigma})\]
    for $\omega\in \mathcal{E}_+(m)$, $c_1,\ldots, c_m\in \widetilde N^\ast(X, R)$ and for any regular simplex, $\sigma:\Delta=\Delta^{j_0}\ast\ldots\ast\Delta^{j_n}\to X$. The functoriality of the $\mathcal{E}_+$-algebra structure on a tensor product and the functoriality of the association $X\to \widetilde N^\ast(X, R)$ ensures that the $\mathcal{E}_+$-algebra structure is functorial in $X$. 
    
    We're left with checking that this structure preserves the perverse degree. Note that it is a local notion. Let $\overline{p}_1, \ldots, \overline{p}_m$ be GM-perversities. For $1\leq i \leq m$, we consider elements $c_i\in \widetilde N^\ast_{\overline{p}_i}(X, R)$. We need to check that $\norm{\lambda(\omega; c_1,\ldots,c_m)}\leq \overline{p}_1\oplus\ldots\oplus\overline{p}_m$. Recall that the perverse degree of a cochain $\omega\in N^\ast(X, R)$ along a singular stratum $S$ is given by
    \[\norm{\omega}_S:=\sup\{\norm{\omega_\sigma}_{\codim S} \,|\, \sigma: \Delta \to X \text{ regular such that } \sigma(\Delta)\cap S \neq  \emptyset \}.\]
    Hence, checking that the $\mathcal{E}_+$-algebra structure preserves the perverse degree amounts to checking that for any regular simplex $\sigma:\Delta=\Delta^{j_0}\ast\ldots\ast\Delta^{j_n}\to X$, for any family of elements $g_1,\ldots,g_m\in \widetilde N^\ast(\Delta)$ with $\norm{g_1}\leq \overline{p}_i$ we have
    \[\norm{\widetilde\lambda(\omega; g_1,\ldots,g_m)}\leq \overline{p}_1\oplus\ldots \oplus \overline{p}_m\]
    We take $k\in\{1,\ldots, n\}$ such that $\Delta^{j_{n-k}}\neq \emptyset$. The inclusion $\Delta^{j_{n-k}}\times\{1\}\hookrightarrow c\Delta^{j_{n-k}}$ induces a cochain map $\iota_{n-k}:N^\ast(c\Delta^{j_{n-k}})\to N^\ast(\Delta^{j_{n-k}}\times\{1\})$ and by extension, we get a morphism of cochain complexes
    \[\tilde\iota_{n-k}:\widetilde N^\ast(\Delta)\to N^\ast(c\Delta^{j_0})\otimes\ldots N^\ast(\Delta^{j_{n-k}}\times \{1\})\otimes\ldots \otimes N^\ast(c\Delta^{j_{n-1}}) \otimes N^\ast(\Delta^{j_n}.)\]
    For $1\leq i\leq m$, we decompose $g_i$ in $g_i=\sum_{s=0}^{n_i}g_{i,0}^s\otimes \ldots \otimes g_{i,n}^s$. Similarly, we can write $\tilde\iota_{n-k}(g_i)=\sum_{s=0}^{n_i}g_{i,0}^s\otimes \ldots\otimes \iota_{n-k}(g_{i,n-k}^s)\otimes\ldots \otimes g_{i,n}^s$. Recall that
    \[\norm{g_i}_k=\sup_s\{\Real{g_{i,n-l+1}^s\otimes\ldots\otimes g_{i,n}^s} \text{ such that }g_{i,0}^s\otimes \ldots\otimes \iota_{n-k}(g_{i,n-k}^s)\neq 0 \}.\]
    Using the naturality of the $\mathcal{E}_+$-algebra structure, we have
    \[\tilde\iota_{n-k}(\widetilde\lambda(\omega; g_1,\ldots,g_m))=\widetilde\lambda(\omega; \tilde\iota_{n-k}(g_1),\ldots,\tilde\iota_{n-k}(g_m)).\]
    If there is an element $g_i$ such that $\tilde\iota_{n-k}(g_i)=0$ then $\widetilde\lambda(\omega; \tilde\iota_{n-k}(g_1),\ldots,\tilde\iota_{n-k}(g_m))=0$ and hence, $\norm{\tilde\iota_{n-k}(\widetilde\lambda(\omega; g_1,\ldots,g_m))}=-\infty$. We now suppose that for every $1\leq i\leq m$, $\tilde\iota_{n-k}(g_i)\neq 0$. Using the definition of the $\mathcal{E}_+$-algebra structure on a tensor product (using Berger and Fresse's diagonal map mentioned above), we can see that $\widetilde\lambda(\omega; \tilde\iota_{n-k}(g_1),\ldots,\tilde\iota_{n-k}(g_m))$ is a sum of tensor products whose components are of the following form:
    \begin{itemize}
        \item $\widetilde\lambda_l(\omega_l; g_{1, l}^s,\ldots,g_{m, l}^s))$ when $l\neq n-k$
        \item $\widetilde\lambda_{n-k}(\omega_l; \tilde\iota_{n-k}(g_{1, n-k}^s),\ldots,\tilde\iota_{n-k}(g_{m, n-k}^s))$ else
    \end{itemize}
    where $\widetilde\lambda_l$ denotes the following action of the Barratt-Eccles operad
    \[
    \begin{array}{rrcl}
    \widetilde\lambda_l:&\mathcal{E}_+(m)\otimes \widetilde N^\ast(c\Delta^{j_l})^{\otimes m}&\to& \widetilde N^\ast(c\Delta^{j_l})\\
    ~&\omega\otimes g_{1,l}\otimes\ldots\otimes g_{m,l}     &\mapsto&  \widetilde\lambda_l(\omega; g_{1,l},\ldots,g_{m, l}).
    \end{array}
    \]
    Furthermore, $\Real{\widetilde\lambda_l(\omega_l; g_{1, l}^s,\ldots,g_{m, l}^s))}= \sum_{t=0}^m \Real{g_{t, l}^s}+\Real{\omega_l}\leq \sum_{i=0}^m \Real{g_{i, l}^s}$
    since $\Real{\omega_l}\leq 0$. This implies that
    \[\norm{\widetilde\lambda(\omega; g_1,\ldots,g_m)}_k\leq \sum_{i=1}^m \norm{g_i}_k\]
    and hence, we have
    \[\norm{\widetilde\lambda(\omega; g_1,\ldots,g_m)}\leq \sum_{i=1}^m \norm{g_i}\leq \overline{p}_1\oplus\ldots\oplus\overline{p}_m.\]
    Furthermore, using the Leibniz rule on the differential, we notice that
    \[\norm{d(\widetilde\lambda(\omega; g_1,\ldots,g_m))}\leq \max\left(\sum_{i=1}^m \norm{g_i}, \sum_{\substack{i=1 \\ i\neq j}}^m \norm{g_i}+\norm{d(g_j)} \text{ for } 1\leq j\leq m\right)\leq \overline{p}_1\oplus\ldots\oplus\overline{p}_m\]
    if $\norm{d(g_i))}\leq \overline{p}_i$. This proves that the Barratt-Eccles action on the blown-cochain complex preserves the perverse degree, hence, $\widetilde N^\ast_\bullet(X, R)$ is a perverse $\mathcal{E}_+$-algebra.
\end{proof}
    
        \subsection{Proof of Theorem \ref*{thm:BV-Hoch-bup}}\label{subsec:thm_BV_Hoch_bup}
        In this subsection, we explain why Theorem \ref{thm:BV-Hoch-bup} is true. The procedure in similar to the one given in subsection \ref{subsec:Thm_BV_sing_proof}. We recall the statement of Proposition \ref{prop:algBE_envalg} which was given in the first section and prove that we have a quasi-isomorphism when $\mathcal{C}=Ch(\mathbb{F})^{\pGM}$.
        \begin{proposition*}[\textbf{\ref*{prop:algBE_envalg}}]
        Let $\mathcal{C}$ be a symmetric monoidal category over $Ch(\mathbb{F})$. If an object $A$ in $\mathcal{C}$ is an $\mathcal{E}_+$-algebra then
        \begin{enumerate}[label=\roman*)]         
            \item we have a morphism $\env(\overline{A})\twoheadrightarrow A$ of $\env(\overline{A})$-modules,
            \item there is a morphism $A\otimes A^{op}\to \env(\overline{A})$,
            \item and there exists an isomorphism $\env(\overline{A}) \simeq \env(\overline{A})^{op}$
        \end{enumerate}
        where the last two morphisms are morphism of $As$-algebras in $\mathcal{C}$. Furthermore, if $\mathcal{C}=Ch(\mathbb{F})$ or $\mathcal{C}=Ch(\mathbb{F})^{\pGM}$, the morphism of i) is a quasi-isomorphism. 
        \end{proposition*}
        \begin{proof}
        Recall that the map of i) is given by
        \[\begin{array}{lccc}
        f_\bullet:&\env(\overline{A})_\bullet & \to & A_\bullet \\
         ~&\sigma(t, a_1, \ldots, a_k)& \mapsto & \sigma(1_A, a_1, \ldots, a_k)
        \end{array}\]
        where $\sigma(1_A, a_1, \ldots, a_k) \in A_{\overline{p}_1\oplus\ldots\oplus \overline{p}_k}$ if $a_i\in A_{\overline{p}_i}$ for $1\leq i \leq k$.
        If $\sigma\in \mathcal{E}(k+1)$, we have $f(\sigma(t, a_1, \ldots, a_k))=F_k(\sigma)(a_1, \ldots, a_k)$ with $F_k$ being the composite 
        \[\mathcal{E}(k+1)\hookrightarrow \mathcal{E}_+(k+1)\otimes \mathcal{E}_+(0)\otimes \mathcal{E}_+(1)\otimes \ldots \otimes \mathcal{E}_+(1) \xrightarrow{\gamma} \mathcal{E}_+(k)=\mathcal{E}(k)\]
        where $\gamma$ refers to the composition product in $\mathcal{E}_+$. The morphism $f$ admits a right inverse 
        \[\begin{array}{lccc}
        g_\bullet:& A_\bullet & \to &\mathcal{U}_{\mathcal{E}}(\overline{A})_\bullet \\
         ~&a& \mapsto & (1,2)(t,a). 
        \end{array}\]
        We're going to show that $g \circ f$ and $\id_{\env(\overline{A})}$ are chain homotopic when $\mathcal{C}=Ch(\mathbb{F})^{\pGM}$. We only prove the result for free $\mathcal{E}_+$-algebras. The general result is deduced from the fact that any $\mathcal{E}_+$-algebra can be expressed as the coequalizer of free algebras. We suppose that $A_\bullet=\mathcal{E}_+(V_\bullet)$ for some perverse chain complex $V_\bullet$. We have $\mathcal{E}_+(V_\bullet):=\oplus_{k\geq 0}(\mathcal{E}_+(k)\otimes V_\bullet^{\otimes k})_{\mathfrak{S}_k}=\mathcal{E}(0)\oplus\mathcal{E}(V_\bullet)$. Let $a=\sigma(v_1, \ldots, v_k)\in A$, we have $g(\sigma(v_1, \ldots, v_k))=G_k(\sigma)(t, v_1, \ldots, v_k)$ where 
        \[
        \begin{array}{lccc}
             G_k:& \mathcal{E}(k) &\to& \mathcal{E}(k+1)  \\
             ~&  \sigma &\mapsto& (1, 2)\circ_2 \sigma.
        \end{array}
        \]
        We have $F_k\circ G_k=\id_{\mathcal{E}(k)}$ and there exists a chain homotopy $H:\mathcal{E}(k+1)\to \mathcal{E}(k+1)$ (of degree +1) such that
        \[d_{\mathcal{E}(k+1)}\circ H + H\circ d_{\mathcal{E}(k+1)} = G_k\circ F_k - \id_{\mathcal{E}(k+1)}.\]
         For every $\overline{p}\in\pGM$, we can lift $H$ into a map $h_{\overline{p}}:\env(\overline{A})_{\overline{p}}\to \env(\overline{A})_{\overline{p}}$ where $h_{\overline{p}}(\sigma(t, a_1, \ldots, a_k)):=H(\sigma)(t, a_1, \ldots, a_k)$ for $\sigma(t, a_1, \ldots, a_k)\in\env(\overline{A})_{\overline{p}}$. Note that by construction, this family of maps indexed by $\pGM$ commutes with the structure morphisms. 
         Using the equality given above we can show that $d_{\env(\overline{A})}\circ h_{\overline{p}} + h_{\overline{p}}\circ d_{\env(\overline{A})} = (g\circ f)_{\overline{p}} - \id_{{\env(\overline{A})}}$. This proves that $f_\bullet$ and $g_\bullet$ are chain homotopy equivalences, whence they are quasi-isomorphisms of perverse chain complexes.
        \end{proof}
        As a corollary, we get the next result which is analogous to corollary \ref{coro:algBE_cochain_envalg}.
        \begin{corollaire}
        Let $\mathbb{F}$ be a field and $n\in\mathbb{N}$. There exists a functor $E: (\nstrat)^{op} \to pDGA$ from stratified spaces of formal dimension $n$ to the category of pDGAs such that for any stratified space $X$ the following properties are verified:
        \begin{enumerate}[label=\roman*)]
            \item we have a quasi-isomorphism of left $E(X)$-modules $E(X)\xtwoheadrightarrow{\simeq} \widetilde N^\ast_\bullet(X; \mathbb{F})$,            
            \item there is a pDGA morphism $\widetilde N^\ast_\bullet(X; \mathbb{F})\otimes \widetilde N^\ast_\bullet(X; \mathbb{F})^{op}\to E(X)$,
            \item and there exists an isomorphism of pDGAs $E(X) \simeq E(X)^{op}$.
        \end{enumerate}
        \end{corollaire}
        \begin{proof}
            For $X$ a stratified space, we set $E(X)=\env(\overline{\widetilde N^\ast_\bullet(X; \mathbb{F})})$. The functorialy of $E$ is deduced from the functorialy of the blown-up intersection cochain complex and the enveloping algebra. The rest of the result follows from the previous proposition.
        \end{proof}
        
        Theorem \ref{thm:BV-Hoch-bup} follows from Poincaré duality (Proposition \ref{prop:dual-qiso-bup}), the structure of an $\mathcal{E}_+$-algebra on the blown-up intersection cochains of a stratified space (Proposition \ref{prop:BE_alg_bup_cochains}) and the next result.
        \begin{proposition}\label{prop:BV_HH_pAlg_BE}
            Let $A_\bullet$ be an algebra over the Barratt-Eccles operad $\mathcal{E}_+$ in $Ch(\mathbb{F})^{\pGM}$. If there exists a quasi-isomorphism of $A$-modules between $A$ and its linear dual $DA$ then one can endow the Hochschild cohomology $HH^\ast_\bullet(A)$ with a perverse Batalin-Vilkovisky algebra structure. 
        \end{proposition}
        We omit the proof of this result as it is identical to the one of Proposition \ref{prop:BV_HH_Alg_BE}. We'll need just the following preliminary results which are analogous to statements given in subsection \ref{subsec:Thm_BV_sing_proof}. The demonstrations can be found in \cite[Section 3]{Raz23}.

        \begin{defi}
            We say that a pDGA $A_\bullet$ is a \emph{perverse derived Poincaré duality algebra} (pDPDA) if it isomorphic to its dual $DA[k]_\bullet$ in the derived category of pDG $A_\bullet$-bimodules up to a shift in degree by $k\in\mathbb{Z}$.
        \end{defi}
        
        \begin{rem}\label{rem:pDPDA}
            In other words, if $A_\bullet$ is pDPDA then there exist, in the category of a $A_\bullet$-bimodules, a cofibrant approximation $q_A: P_\bullet\xrightarrow{\simeq} A_\bullet$ and a quasi-isomorphism of $c:P_\bullet\xrightarrow{\simeq} DA[k]_\bullet$ that fit in the following zigzag:
            \[A_\bullet \xleftarrow[\simeq]{q_A} P_\bullet \xrightarrow[\simeq]{c} DA[k]_\bullet.\]
            In particular, we can take the bar construction $P_\bullet=\mathbb{B}(A)_\bullet$ (when it is cofibrant as a pDG $A_\bullet$-bimodule). We have the following commutative diagram where the vertical maps are isomorphisms
            \[\adjustbox{scale=0.8}{\begin{tikzcd}[column sep = huge]
        	{\Hom_{A^e}(\mathbb{B}(A),A)_\bullet} & {\Hom_{A^e}(\mathbb{B}(A),\mathbb{B}(A))_\bullet} & {\Hom_{A^e}(\mathbb{B}(A),DA[k])_\bullet} \\
        	{\Hom_{Ch(R)^{\pGM}}(T(s\overline{A}),A)_\bullet} & {\Hom_{Ch(R)^{\pGM}}(T(s\overline{A}),\mathbb{B}(A))_\bullet} & {\Hom_{Ch(R)^{\pGM}}(T(s\overline{A}),DA[k])_\bullet}
        	\arrow["{\Hom_{A^e}(\mathbb{B}(A),q_A)}"', from=1-2, to=1-1]
        	\arrow["{\Hom_{A^e}(\mathbb{B}(A),c)}", from=1-2, to=1-3]
        	\arrow["{\Hom_{Ch(R)^{\pGM}}(T(s\overline{A}),c)}"', from=2-2, to=2-3]
        	\arrow["{\Hom_{Ch(R)^{\pGM}}(T(s\overline{A}),q_A)}.", from=2-2, to=2-1]
        	\arrow["\simeq"{marking}, draw=none, from=2-1, to=1-1]
        	\arrow["\simeq"{marking}, draw=none, from=2-3, to=1-3]
        	\arrow["\simeq"{marking}, draw=none, from=2-2, to=1-2]
            \end{tikzcd}}\]
            Since $Ch(R)^{\pGM}$ is a monoidal model category, by Remark \ref{rem:quillen_adj_pch}, the bottom arrows are all quasi-isomorphisms. Hence, when we take homology, we get an isomorphism:
            \[HH^\ast_\bullet(A) \xleftarrow[\simeq]{HH^\ast(A, q_A)} HH^\ast_\bullet(A, P) \xrightarrow[\simeq]{HH^\ast(A,c)} HH^\ast_\bullet(A, DA)[k].\]
            Notice that $[q_A]$, which is the unit of the left $HH^\ast_\bullet(A)$-module $HH^\ast_\bullet(A)$, is sent to $[Id_P]$ and finally to $[c]$. Hence, this isomorphism corresponds to the action of $HH^\ast_\bullet(A)$ on $[c]$ (given at the end of subsection \ref{subsub:Hoch_def_prop}). 
        \end{rem}
        \begin{rem}
            In what follows, we will omit the shift $k$ to lighten the notation. 
        \end{rem}

        Let $ev(1_A):\Ext_A(A, DA)_\bullet\to H(DA)_\bullet$ be the evaluation at the unit $1_A$ of $A_\bullet$ and let $i_A:A_\bullet\hookrightarrow (A\boxtimes A^{op})_\bullet$ be the inclusion in the first factor. We denote by $\eval$ the composite
            \[eval:HH^\ast_\bullet(A,DA):=\Ext^\ast_{A^e}(A,DA)_\bullet\xrightarrow{\Ext_{i_A}(A, DA)} \Ext_A^\ast(A, DA)_\bullet\simeq H(DA)_\bullet.\]
       
        \begin{proposition}
        Let $A$ be a pDGA and let $c\in HC^\ast_\bullet(A,DA)$ be a cocycle such that
        the morphism of $H(A)_\bullet$-modules
        \begin{align*}
            H(A)_\bullet &\xrightarrow{\simeq} H(DA)_\bullet \\
            a &\mapsto a.\eval([c])    
        \end{align*}
        is an isomorphism.
        Then $c$ is a quasi-isomorphism and hence, $A_\bullet$ is a perverse derived Poincaré dualité algebra.
        \end{proposition}
  
        \begin{proposition}
        Let $A_\bullet$ be a perverse derived Poincaré duality algebra and let $\phi:HH^\ast_\bullet(A)\to HH^\ast_\bullet(A,DA)$ be the associated isomorphism of $HH^\ast_\bullet(A)$-bimodules. We define a map $\Delta: HH^\ast_\bullet(A) \to HH^{\ast-1}_\bullet(A)$ by setting 
        \[\Delta=\phi^{-1}\circ B^\vee \circ \phi.\]
        If $\Delta(1)=0$, then the perverse Gerstenhaber algebra $HH^\ast_\bullet(A)$ equipped with $\Delta$ becomes a perverse BV-algebra.
        \end{proposition}        

\appendix
\section{Intersection homology theories}\label{app:inter_hom}
    In this section, we're going to present some intersection homology theories, namely the original theory introduced by Goresky and MacPherson \cite{GM80} and the \emph{blown-up intersection cohomology} studied by Chataur, Saralegui and Tanré. We recall how they are defined and give their main properties. In the last section, they will be studied as \emph{perverse chain complexes} (i.e. functors from a poset of perversities into the category of chain complexes) and some of their properties are given in this formalism. 

    \subsection{Goresky MacPherson intersection homology}\label{subsect: interhom}
    Let's explain what kind of spaces with singularities we will be studying throughout this paper. We will follow \cite[Section 1]{CST18BUP-Alpine}, another good reference for these notions is \cite[Chapter 2]{Fri20}.
    \begin{defi}\label{def:filtered_space}
    A \emph{filtered space} is a Hausdorff space $X$ equipped with a filtration by closed subspaces 
    \[\emptyset = X_{-1} \subseteq X_0 \subseteq \ldots \subseteq X_n=X,\]
    such that $X_n\setminus X_{n-1}$ is non-empty. The \emph{formal dimension} of $X$ is $\dim(X)=n$. 
    
    The \emph{strata} of $X$ are the non-empty connected components of $X_i\setminus X_{i-1}$. 
    The \emph{formal dimension} of a stratum $S\subset X_i\setminus X_{i-1}$ is $\dim(S)=i$. Its \emph{formal codimension} is $\codim(S)=\dim(X)-\dim(S)$.
    The \emph{regular strata} are the strata of dimension $n$ and the other strata are \emph{singular}. The set of strata of $X$ is denoted $\mathcal{S}_X$. 
    \end{defi}        
    
    Continuous maps between filtered spaces will not necessarily induce a morphism on intersection homology, this is why we consider filtered spaces which verify an additional property.

    \begin{defi}
        A filtered space  $X$ is \emph{stratified} if it verifies the \emph{Frontier condition}:
        \begin{center}
        \emph{For any pair $S,T \in \mathcal{S}_X$, if $S\cap \overline{T}\neq \emptyset$ then $S\subset T$.}
        \end{center}
        A continuous map $f:X\to Y$ between stratified spaces is \emph{stratified} if for every $S\in \mathcal{S}_X$ there exists a stratum $T\in \mathcal{S}_Y$ such that
        \[f(S)\subset T \text{ and } \codim(S)\geq \codim(T).\]
        A statified map is a \emph{stratified homeomorphism} if it is an homeomorphism and its inverse map is also stratified. \\        
        We denote by $\strat$ the category whose objects are stratified spaces and whose morphisms are stratified maps.
    \end{defi}
    \begin{defi}
        Let $f,g:X\to Y$ be stratified maps between stratified spaces. They are \emph{stratified homotopic} if there exists a stratified map $H:[0,1]\times X \to Y$ such that $f=H|_{\{0\}\times X}$ and $g=H|_{\{1\}\times X}$. The map $H$ is called a \emph{stratified homotopy}.

        Two stratified spaces $X$ and $Y$ are \emph{stratified homotopy equivalent} if there exist stratified maps $f:X\to Y$ and $g:Y\to X$ such that
        \begin{itemize}
            \item $f\circ g$ and $g\circ f$ are stratified homotopic to $\id_Y$ and $\id_X$ respectively,
            \item for each stratum $S\in \mathcal{S}_X$, we have $\codim(f(S))=\codim(S)$,
            \item and for each stratum $T\in \mathcal{S}_Y$, we have $\codim(g(T))=\codim(T)$.
        \end{itemize}
        We say that the maps $f$ and $g$ are \emph{stratified homotopy equivalences}.
    \end{defi}
    Filtered spaces are enough to define intersection homology but one needs additional structure in order to get Poincaré duality. 
    \begin{defi}
        A filtered space $X$ of dimension $n$, with no codimension one strata, is an \emph{$n$-dimensional pseudomanifold} if for any $i\in\{0,1,\ldots, n\}$, $X_i\setminus X_{i-1}$ is an $i$-dimensional manifold. Furthermore, for any $i\in\{0,1,\ldots, n-1\}$ and for every $x\in X_i\setminus X_{i-1}$, there exist
        \begin{itemize}
            \item an open neighborhood $V$ of $x$ in $X$ endowed with the induced filtration,
            \item an open neighborhood $U$ of $x$ in $X_i\setminus X_{i-1}$,
            \item a \emph{link} i.e. a compact pseudomanifold $L$ of dimension $n-i-1$, whose open cone $\mathring{c}L$ is endowed with the filtration $(\mathring{c}L)_i=\mathring{c}L_{i-1},$
            \item a homeomorphism, $\phi: U \times \mathring{c}L \to V$ such that
                \begin{itemize}
                    \item $\phi(u,v)=u$ for any $u\in U$ and $v$ the apex of $\mathring{c}L$,
                    \item $\phi(U\times \mathring{c}L_j)=V\cap X_{i+j+1}$, for all $j\in\{0, \ldots, n-i-1\}.$
                \end{itemize}
        \end{itemize}
    \end{defi}    
    A pseudomanifold of dimension 0 is a disjoint union of points. Since the link's formal dimension is lower than the formal dimension of $X$, the notion of pseudomanifold is well defined. Such a recursive definition makes it possible to prove some results by induction. 
    \begin{rem}
        A pseudomanifold (more generally, a \emph{CS set}) is a stratified space, see \cite[Theorem G]{CST18ration}.
    \end{rem}
    One can endow every filtered space with a poset of \emph{perversities}. The \emph{intersection chain complex} consists of chains that verify some conditions with respect to these perversities. We now give Goresky and MacPherson's original definition of a \emph{perversity}. 
    \begin{defi}
    Let $n\in \mathbb{N}$, a \emph{Goresky-MacPherson $n$-perversity} (GM-perversity) is a map 
    \[\overline{p}:\{0, 1, 2\ldots, n\}\to \mathbb{N}\] 
    such that $\overline{p}(0)=\overline{p}(1)=\overline{p}(2)=0$ and 
    \[\overline{p}(i)\leq \overline{p}(i+1)\leq \overline{p}(i)+1\] 
    for any $i\in \{1,\ldots, n-1\}$.
    \end{defi}
    \begin{rem}\label{rem:genperv}
        Let $X$ be a filtered space of formal dimension $n$. Note that a Goresky-MacPherson perversity induces a map
        \[\overline{p}:\mathcal{S}_X\to \mathbb{N}\]
        by setting for a singular stratum $S$, $\overline{p}(S)=\overline{p}(\codim(S))$ (and 0 on the regular strata). A \emph{perversity on $X$} will refer to a GM $n$-perversity if  $X$ is a filtered space of formal dimension $n$.\\
        In MacPherson's work \cite{MacP91}, a perversity is defined as a function
        \[\overline{p}:\mathcal{S}_X\to \mathbb{Z}\]
        that takes the value $0$ on the regular strata. This general notion of perversity has been used in papers dealing with non-GM intersection homology. (\cite{Fri20}, \cite{CST18BUP-Alpine} for example).
    \end{rem}
    
    \textbf{Poset of GM-perversities:} For $n\in \mathbb{N}$, we consider the set of GM $n$-perversities, $\pGM$. It is a poset. The partial order is given by $\leq$: 
    \[\overline{p}\leq \overline{q} \Longleftrightarrow \overline{p}(i)\leq \overline{q}(i), ~\forall i\in \{0,\ldots, n\}.\]
    It can be seen as a category where there is at most one morphism between any pair of perversities.\\
    Notice that it has a minimal element the \emph{zero perversity} $\overline{0}$ and a maximal element, the \emph{top perversity} $\overline{t}$, given by 
    \[\overline{t}(0)=\overline{t}(1)=0 \text{ and } \overline{t}(i)=i-2 \text{ for } i\in \{2,\ldots, n\}.\]
    
    \textbf{Partial symmetric monoidal structure:} In general, the point-wise sum of two GM-perversities is not a GM-perversity but one can define a partial symmetric monoidal structure on $\pGM$. For $\overline{p}, \overline{q}\in \pGM$ such that $\overline{p}+\overline{q}\leq \overline{t}$, we denote by $\overline{p}\oplus\overline{q}$ the smallest GM-perversity which is greater or equal to $\overline{p}+\overline{q}$. Dually, when $\overline{p}\leq \overline{q}$, let $\overline{q}\ominus \overline{p}$ be the biggest GM-perversity which is lower or equal to $\overline{q}-\overline{p}$. \\
    The \emph{dual perversity} (or \emph{complementary perversity}) of $\overline{p}$ is $D\overline{p}:=\overline{t}\ominus\overline{p}$. Note that it is exactly $\overline{t}-\overline{p}$. 
    \\ When $\oplus$ and $\ominus$ are defined, they satisfy the properties expected from a closed symmetric monoidal category, see \cite[Section 1]{Hov09}.
    
    Let $X$ be a filtered space of formal dimension $n$ and $\overline{p}\in \pGM$. We will follow Chataur, Saralegui and Tanré's exposition of intersection homology given in \cite{CST18ration}. They present $I^{\overline{p}}C_\ast(X;R)$, the \emph{$\overline{p}$-intersection chain complex} of $X$ and define the \emph{blown-up complex of $\overline{p}$-intersection cochains} $\widetilde{N}^\ast_{\overline{p}}(X ;R)$ (also known as the \emph{Thom-Whitney complex}). The reader may refer to \cite{CST18BUP-Alpine}, \cite{CST18ration} and \cite{Fri20} for explicit constructions, we're just going to give the main properties of these complexes.

    \begin{proposition}
        Let $\overline{p}$ be a GM $n$-perversity. There is a functor
        \[\begin{array}{rrcl}
        I^{\overline{p}} C_\ast: & \nstrat & \to & Ch(R) \\
        ~& X & \mapsto & I^{\overline{p}} C_\ast(X ; R). \\
        \end{array}\]   
        from the category of stratified spaces of formal dimension $n$ to the category of chain complexes. Furthermore, the homology of $I^{\overline{p}} C_\ast(X;R)$, called \emph{$\overline{p}$-intersection homology} and denoted $I^{\overline{p}} H_\ast(X;R)$, verifies the following properties.
        \begin{description}[style=unboxed,leftmargin=0cm]
            \item[Mayer-Vietoris] If $\{U, V\}$ is an open cover of $X$, then there is a long exact sequence
            \[\ldots \to I^{\overline{p}}H_\ast(U\cap V) \to I^{\overline{p}}H_\ast(U)\oplus I^{\overline{p}}H_\ast(V)\to I^{\overline{p}}H_\ast(X)\to I^{\overline{p}}H_{\ast-1}(U\cap V)\to \ldots\]
            \item[Cone formula] If $X$ is compact, we have the \emph{cone formula}. For any $k\in \mathbb{N}$,
            \[I^{\overline{p}}H_k(\mathring{c}X; R)=\begin{cases}
                I^{\overline{p}}H_k(X; R) & \text{if }k\leq n-\overline{p}(n+1)-1, \\
                0 & \text{if } k \geq n-\overline{p}(n+1).
            \end{cases}\]
            \item[Topological invariance] If $X$ and $Y$ are two pseudomanifold of formal dimension $n$ which are homeomorphic (not necessarily stratified homeomorphic) then for every $\overline{p}\in\pGM$ we have an isomorphism
            \[I^{\overline{p}} H_\ast(X;R) \simeq I^{\overline{p}} H_\ast(Y;R).\]
            \item[Invariance under stratified homotopy equivalence] Let $f:X\to Y$ be a stratified homotopy equivalence between $X\in \nstrat$ and $Y\in \mstrat$. Let $\overline{p}$ and $\overline{q}$ be perversities of rank $n$ and $m$ respectively such that for any $S\in \mathcal{S}_X$, we have $\overline{p}(\codim(S))=\overline{q}(\codim(T))$  when $f(S)\subset T$.
            Then, we have an isomorphism
            \[I^{\overline{p}} H_\ast(X;R) \simeq I^{\overline{q}} H_\ast(Y;R).\]
        \end{description}
    \end{proposition}
    \begin{defi}
        Let $X$ be a filtered space of formal dimension $n$ and $\overline{p}\in \pGM$. By taking the linear dual of the $\overline{p}$-intersection chain complex, we define the \emph{dual $\overline{p}$-intersection cochain complex}, $I_{\overline{p}}C^\ast(X;R):=\Hom_R(I^{\overline{p}}C_\ast(X;R), R)$ and its homology is the \emph{dual $\overline{p}$-intersection cohomology} $I_ {\overline{p}}H^\ast(X;R)$.
    \end{defi}

    We can now state \emph{Poincaré duality}, which is the main reason intersection homology was introduced. 
    \begin{thm}[\textbf{(Poincaré duality)} - {\cite[Theorem 8.2.4]{Fri20}}]
        Let $X$ be an $n$-dimensional, compact and oriented pseudomanifold. Then, for every perversity $\overline{p}$ on $X$ there exists an isomorphism
        \[I_{\overline{p}}H^\ast(X; \mathbb{F}) \xrightarrow{\simeq} I^{D\overline{p}}H_{n-\ast}(X; \mathbb{F}).\]
    \end{thm}
    \subsection{Blown-up intersection cohomology}\label{subsect:bupinter}
    There exists several variants to intersection (co)homology. For instance, we have non-GM intersection (co)homology theories which are presented in \cite{Fri20} and the blown-up intersection cohomology which we now consider. Note that the blown-up cochain complex appears in several papers  \cite{CST18BUP-Alpine}, \cite{CST18ration},  \cite{CST18DP}, \cite{CST20Deligne} and \cite{ST20} to name a few. Our main reference for this section is \cite{CST18BUP-Alpine}.
    
        \subsubsection{Blown-up intersection cochain complex}
    For the rest of this subsection, let $X$ be a filtered space of formal dimension $n$ and let $\{X_i\}_{-1\leq i \leq n}$ be the associated filtration. We first need to define filtered simplices.
    \begin{defi}
        Let $k\in \mathbb{N}$. A \emph{Euclidean simplex} of dimension $k$ is the convex hull of $k+1$ affinely independent points in $\mathbb{R}^{k+1}$. The \emph{standard $k$-simplex} $\Delta^k$ is given by
        \[\Delta^k:=\left\{(x_0, \ldots, x_k)\text{ such that } \sum_{i=0}^k x_i=1, x_i\geq 0\right\}\subset \mathbb{R}^{k+1}.\]
        It can also be seen as the convex hull $\Delta^k=[e_0, \ldots, e_k]$ of the points $\{e_i\}_{0\leq i\leq k}$ where $e_i=(0, \ldots, 1, \ldots, 0)$ is the point whose $i^{\text{th}}$ coordinate is a 1.
    \end{defi}
    \begin{defi}
        Let $k\in \mathbb{N}$. A \emph{filtered $k$-simplex} of $X$ is a continuous map $\sigma:\Delta^k\to X$ such that $\sigma^{-1}(X_i)$ is a face of $\Delta^k$ or is empty for all $i\in\{0,\ldots,n\}$.
    \end{defi}      
    We now give a characterization of filtered simplicies which will be useful in what follows. We consider a filtered space $X$ of formal dimension $n$ and let $\{X_i\}_{-1\leq i \leq n}$ be the associated filtration.
    
    \begin{proposition}
         A continuous map $\sigma:\Delta^k\to X$ is a filtered $k$-simplex of $X$ if and only if there exists a decomposition of $\Delta^k$ into joins of simplices (eventually empty) 
            $\Delta^k= \Delta_0\ast\ldots \ast\Delta_n$ such that for all $i\in\{0,\ldots,n\}$, $\Delta_i$ is a Euclidean simplex and $\sigma^{-1}X_i=\Delta_0\ast \Delta_1 \ast \ldots \ast \Delta_i$.
    \end{proposition}    
    
    \begin{defi}
    Let $\Delta$ be a Euclidean simplex. A \emph{regular filtered simplex} of $X$ is a filtered simplex $\sigma: \Delta \to X$ such that the decomposition of $\Delta$ into joins of simplices $\Delta= \Delta_0\ast\ldots \ast\Delta_n$ is such that $\Delta_n\neq \emptyset$.
    \end{defi}    
    \begin{defi}
    Let $X$ be a simplicial complex. We denote by $(N_\ast(X), \partial)$ the \emph{simplicial chain complex} of $X$ with coefficient in $R$. Its degree $k$ component is given by 
    \[N_k(X)=\left\{\sum_{\text{finite}} \lambda_i\sigma _i \,|\, \lambda_i\in R \text{ and } \sigma_i \text{ $k$-simplices}\right\}.\]
    If $\sigma=[v_0, \ldots, v_k]$ is a $k$-simplex, its differential is defined by
    \[\partial(\sigma)=\sum_{i=0}^k (-1)^i [v_0, \ldots, v_{i-1}, v_{i+1}, \ldots, v_k].\]
    The \emph{simplicial cochain complex} $(N^\ast(X), \delta)$ associated to $X$ is obtained by setting for any $k\in \mathbb{N}$, $N^k(X)=\Hom_R(N_ k(X), R)$ and with $\delta$ the dual map of $\partial$.
    \end{defi}
    \begin{defi}
    Let $\Delta= \Delta_0\ast\ldots \ast\Delta_n$ be a regular Euclidean simplex. The \emph{blown-up cochain complex} $(\widetilde{N}^\ast(\Delta), d)$ associated to $\Delta$ is the cochain complex defined as the following tensor product of complexes
    \[\widetilde{N}^\ast(\Delta)=N^\ast(c\Delta_0)\otimes \ldots \otimes N^\ast(c\Delta_{n-1})\otimes N^\ast(\Delta_n).\]
    The degree $k\in \mathbb{N}$ component of this cochain complex is
    \[\widetilde{N}^k(\Delta)= \bigoplus_{i_0+\ldots+i_n=k} N^{i_0}(c\Delta_0)\otimes \ldots \otimes N^{i_{n-1}}(c\Delta_{n-1})\otimes N^{i_n}(\Delta_n).\]
    The differential $d$ is given for $f_0\otimes \ldots \otimes f_n\in \widetilde{N}^\ast(\Delta)$ by
    \[d(f_0\otimes \ldots \otimes f_n)=\sum_{i=0}^n (-1)^{\eps_i} f_0\otimes \ldots \otimes\delta(f_i)\otimes \ldots \otimes f_n\]
    with $\eps_i=\sum_{0\leq j <i} \Real{f_j}$.
    \end{defi}

    We're going to define a perverse degree for elements in $\widetilde{N}^\ast(\Delta)$ for $\Delta= \Delta_0\ast\ldots \ast\Delta_n$. To do so, we need to explain a few things. Let $k\in \{0, \ldots, n-1\}$ and $\omega \in \widetilde{N}^\ast(\Delta)$. If $\Delta_k\neq \emptyset$, we denote by $\omega_k$ the restriction of $\omega$ to
    \[N^\ast(c\Delta_0)\otimes \ldots \otimes N^\ast(\Delta_{k}\times \{1\}) \otimes \ldots \otimes N^\ast(c\Delta_{n-1})\otimes N^\ast(\Delta_n).\]
    The term $\omega_k$ can be written as $\omega_k=\sum_l \omega_k'(l)\otimes \omega_k''(l)$ where
    \begin{itemize}
        \item $\omega_k'(l)\in N^\ast(c\Delta_0)\otimes \ldots \otimes N^\ast(\Delta_{k}\times \{1\})$
        \item and $\omega_k''(l)\in N^\ast(c\Delta_{k+1}) \otimes \ldots \otimes N^\ast(c\Delta_{n-1})\otimes N^\ast(\Delta_n)$.
    \end{itemize}
    We denote by $\Real{\omega_k''(l)}$ the degree of $\omega_k''(l)$ in $N^\ast(c\Delta_{k+1}) \otimes \ldots \otimes N^\ast(c\Delta_{n-1})\otimes N^\ast(\Delta_n)$.
    \begin{defi}
        Let $\Delta=\Delta_0\ast\ldots\ast \Delta_n$ be a regular Euclidean simplex and $\omega \in \widetilde{N}^\ast(\Delta)$. For $k\in \{1, \ldots, n\}$, we define the \emph{$k$-perverse degree} by
        \[\norm{\omega}_k:=\sup_l\{\Real{\omega_{n-k}''(l)} \text{ such that } \omega_{n-k}'(l)\neq 0\}.\]
        We set $\norm{\omega}_k=-\infty$ if $\omega_{n-k}=0$ or $\Delta_{n-k}=\emptyset$. 
    \end{defi} 

    \begin{defi}
    The \emph{blown-up cochain complex} of a filtered space $X$ is the cochain complex $\widetilde{N}^\ast(X;R)$ made of elements $\omega$ that associate to each regular simplex $\sigma:\Delta\to X$ an element $\omega_\sigma \in \widetilde{N}^\ast(\Delta)$. The differential of $\omega$, $d\omega$ is defined by $(d\omega)_\sigma = d (\omega_\sigma)$. \\
    The \emph{perverse degree of $\omega$ along a singular stratum $S$} is given by
    \[\norm{\omega}_S:=\sup\{\norm{\omega_\sigma}_{\codim S} \,|\, \sigma: \Delta \to X \text{ regular such that } \sigma(\Delta)\cap S \neq  \emptyset \}.\]
    We get a map $\norm{\omega}: \mathcal{S}_X \to \mathbb{N}$ by setting $\norm{\omega}_S=0$ for any regular stratum $S$.
    \end{defi}

    \begin{defi}
    A cochain $\omega\in\widetilde{N}^\ast(X ;R)$ is \emph{$\overline{p}$-allowable} if $\norm{\omega}\leq \overline{p}$ and of \emph{$\overline{p}$-intersection} if $\omega$ and $d\omega$ are $\overline{p}$-allowable. The \emph{blown-up complex of $\overline{p}$-intersection cochains} is denoted $\widetilde{N}^\ast_{\overline{p}}(X ;R)$. Its homology is called the \emph{blown-up $\overline{p}$-intersection cohomology} of $X$ and is denoted $\mathscr{H}^\ast_{\overline{p}}(X ;R)$.   
    \end{defi} 
    The blown-up intersection cohomology can be defined for more general perversities. It coincides with intersection cohomology as it was defined by Goresky and MacPherson. Furthermore, it corresponds to the dual intersection cohomology when we work on a field and for GM-perversities.
    \begin{thm}[{\cite[Theorem B]{CST18ration}}]\label{thm:field_intercoho_iso}
        Suppose that $R$ is a field and let $X$ be a filtered space of formal dimension $n$. For any $\overline{p}\in\pGM$, we have a quasi-isomorphism
        \[\Inter:\widetilde{N}^\ast_{\overline{p}}(X ;R)\to I_{\overline{t}-\overline{p}}C^\ast(X ;R).\]
    \end{thm}
        \subsubsection{Properties of the blown-up intersection cohomology}
    Just like we did for intersection homology, we list the main properties of the blown-up intersection cohomology. A proof of the following results can be found in \cite{CST18BUP-Alpine}.
    \begin{proposition}\label{prop:bup_prop}
        Let $\overline{p}$ be a GM $n$-perversity. There is a functor
        \[\begin{array}{rrcl}
        \widetilde{N}^\ast_{\overline{p}}: & (\nstrat)^{op} & \to & Ch(R) \\
        ~& X & \mapsto & \widetilde{N}^\ast_{\overline{p}}(X ; R). \\
        \end{array}\]    
        Furthermore, the blown-up $\overline{p}$-intersection cohomology, $\mathscr{H}^{\overline{p}}_\ast(X;R)$, verifies the following properties.
        \begin{description}[style=unboxed,leftmargin=0cm]
            \item[Mayer-Vietoris] Suppose that $X$ is paracompact and has an open cover $\{U, V\}$, then there is a long exact sequence
            \[\ldots \to \mathscr{H}_{\overline{p}}^\ast(X;R) \to \mathscr{H}_{\overline{p}}^\ast(U;R)\oplus \mathscr{H}_{\overline{p}}^\ast(V;R)\to \mathscr{H}_{\overline{p}}^\ast(U\cap V; R)\to \mathscr{H}_{\overline{p}}^{\ast+1}(X;R)\to \ldots \]
            \item[Cone formula] If $X$ is a compact $n$-pseudomanifold, we have,
            \[\mathscr{H}^k_{\overline{p}}(\mathring{c}X; R)=\begin{cases}
                \mathscr{H}^k_{\overline{p}}(X; R) & \text{if }k\leq\overline{p}(n+1), \\
                0 & \text{if } k> \overline{p}(n+1)
            \end{cases}\]
            where $\overline{p}$ is a perversity on $\mathring{c}X$.
            \item[Topological invariance] If $X$ and $Y$ are two pseudomanifolds of formal dimension $n$ which are homeomorphic (not necessarily stratified homeomorphic) then for every $\overline{p}\in\pGM$ we have an isomorphism
            \[\mathscr{H}^\ast_{\overline{p}}(X; R) \simeq \mathscr{H}^\ast_{\overline{p}}(Y; R).\]
            \item[Invariance under stratified homotopy equivalence] Let $f:X\to Y$ be a stratified homotopy equivalence between $X\in \nstrat$ and $Y\in \mstrat$. If $\overline{p}$ and $\overline{q}$ are $f$-compatible perversities, then we have an isomorphism
            \[\mathscr{H}^\ast_{\overline{p}}(X; R) \simeq \mathscr{H}^\ast_{\overline{q}}(Y; R).\]
        \end{description}
    \end{proposition}    

        \subsubsection{Cup and cap products}
    We now recall results regarding algebraic structures found on the blown-up intersection complex. The construction of the cup and cap products appear in \cite[Section 4 and 6]{CST18BUP-Alpine}.

    \begin{proposition}[{\cite[Proposition 4.2]{CST18BUP-Alpine}}]\label{prop:cup_bup}
        Let $X$ be a filtered space of formal dimension $n$. For any $\overline{p}, \overline{q}\in \pGM$ and $i,j\in \mathbb{N}$, there exists an associative multiplication
        \[-\cup-:\widetilde{N}^i_{\overline{p}}(X ;R)\otimes \widetilde{N}^j_{\overline{q}}(X ;R)\to \widetilde{N}^{i+j}_{\overline{p}\oplus \overline{q}}(X ;R).\]
        It induces a graded commutative multiplication in homology called \emph{intersection cup product}
        \[-\cup-:\mathscr{H}^i_{\overline{p}}(X ;R)\otimes \mathscr{H}^j_{\overline{q}}(X ;R)\to \mathscr{H}^{i+j}_{\overline{p}\oplus \overline{q}}(X ;R).\]
    \end{proposition}
    
    \begin{proposition}[{\cite[Propostion 6.7]{CST18BUP-Alpine}}]\label{prop:cap_bup}
        Let $X$ be a filtered space of formal dimension $n$. For any $\overline{p}, \overline{q}\in \pGM$ and $i,j\in \mathbb{N}$, there is a well defined \emph{intersection cap product}
        \[-\cap-:\widetilde{N}^i_{\overline{p}}(X ;R)\otimes I^{\overline{q}}C_j(X ;R)\to I^{\overline{p}\oplus \overline{q}}C_{j-i}(X ;R).\]
        It induces a morphism in homology
        \[-\cap-:\mathscr{H}^i_{\overline{p}}(X ;R)\otimes I^{\overline{q}}H_j(X ;R)\to I^{\overline{p}\oplus \overline{q}}H_{j-i}(X ;R).\]
    \end{proposition}
    One of the advantages of the blown-up cochain complex over the dual intersection cochain complex is that we are able to lift Poincaré duality to the (co)chain level. The following result is a consequence of \cite[Theorem 3.1]{ST20}. 
    \begin{thm}[\textbf{- Poincaré duality}]\label{thm:Poincaré}
    Let $X$ be an $n$-dimensional, second countable, compact and oriented pseudomanifold. Then, for every $\overline{p}\in \pGM$ there exists a quasi-isomorphism
    \[\widetilde N^\ast_{\overline{p}}(X;R) \xrightarrow[\simeq]{DP_X} I^{\overline{p}}C_{n-\ast}(X;R)\]
    where $DP_X$ is the cap product with a fundamental cycle $\zeta\in I^{\overline{0}}C_n(X;R)$.
    \end{thm}

    \subsection{The category of perverse chain complexes}
    \subsubsection{Generalities on perverse objects}
    In \cite{Hov09}, Mark Hovey defines perverse modules and chain complexes for Goresky-MacPherson perversities. More generally, one can consider perverse objects in an abelian category. In the following, we fix $n\in \mathbb{N}$ and we consider the poset of GM $n$-perversities which we denoted by $\pGM$. It can be seen as a category where there is at most one morphism between any pair of perversities.
    
    \begin{defi}
    Let $n\in\mathbb{N}$ and  $\mathcal{C}$ be an abelian category. A \emph{perverse object} in $\mathcal{C}$ of \emph{rank} $n$ is a functor \[M:\pGM\to \mathcal{C}.\] The perverse objects in $\mathcal{C}$ of rank $n$ form a category where morphisms are natural transformations. We will denote it $\mathcal{C}^{\pGM}$. In what follows, the rank of an object will be omitted. 
    \end{defi}
    Alternatively, a perverse object is a family of objects $\{M_{\overline{p}}\}_{\overline{p}\in \pGM}$ in $\mathcal{C}$ such that for any pair of GM-perversities $\overline{p} \leq \overline{q}$ we have a morphism in $\mathcal{C}$
    \[\phi_{\overline{p}\leq \overline{q}}: M_{\overline{p}}\to M_{\overline{q}}\]
    which is the identity when $\overline{p}=\overline{q}$ and such that for any perversities $\overline{p}\leq \overline{q}\leq \overline{r}$, we have $\phi_{\overline{p}\leq \overline{r}}=\phi_{\overline{q}\leq \overline{r}}\circ \phi_{\overline{p}\leq \overline{q}}$. We refer to these morphisms as \emph{structure morphisms}.
    \begin{rem}
    Recall that the category of functors from a small category into an abelian category is also abelian \cite[Chapter IX - Prop 3.1]{MacL95} and that limits and colimits are taken point-wise in functor categories. 
    \end{rem}
    \begin{exemple}
    \begin{itemize}
        \item A \emph{perverse module} is a perverse object in the category of $R$-modules. The category of perverse $R$-modules is denoted $(\modu)^{\pGM}$.
        \item A \emph{perverse graded module} $M^\bullet$ is a family of $R$-modules $\{M_i^{\overline{p}}\}_{i\in\mathbb{Z},\, \overline{p}\in \pGM}$ such that for any $\overline{p} \leq \overline{q}$ and $i\in\mathbb{Z}$ we have an $R$-linear map
        \[M^{\overline{p}}_i\to M^{\overline{q}}_i\]
        which is the identity when $\overline{p}=\overline{q}$.\\
        An element $m\in M^{\overline{p}}_i$ is of degree $\Real{m}=i$ and perverse degree $\overline{p}$.
        \item A \emph{perverse chain complex} is a perverse object in the category of chain complexes $Ch(R)$. Equivalently, a perverse chain complex is a chain complex of perverse modules. The category of perverse chain complexes is denoted $(Ch(R))^{\pGM}$.\\
        The degree $i\in \mathbb{Z}$ and perverse degree $\overline{p}\in \pGM$ component of a perverse chain complex $(Z^\bullet, d^\bullet)$ is denoted $Z^{\overline{p}}_i$. Following the convention given at the beginning of this paper, for $i\in \mathbb{Z}$, we set
        \[Z^i_{\overline{p}}:=Z_{-i}^{\overline{p}}.\]
        This defines a \emph{perverse cochain complex} $Z_\bullet$ whose $\overline{p}$-component is the cochain complex $Z^\ast_{\overline{p}}$.
    \end{itemize}
    \end{exemple}
    
    We collect results from \cite{Hov09}. The original statements are given for perverse $R$-modules but one can easily adapt the proofs to perverse objects.   
    \begin{proposition}
    Let $\overline{p}\in\pGM$, there is an exact evaluation functor 
    \[\begin{array}{rrcl}
    Ev_{\overline{p}}: & \mathcal{C}^{\pGM} & \to & \mathcal{C} \\
    ~& M & \mapsto & M_{\overline{p}}. \\
    \end{array}\]   
    The functor $Ev_{\overline{p}}$ possesses a left adjoint $F_{\overline{p}}$ defined for all $N\in \mathcal{C}$ by
    \[F_{\overline{p}}(N)_{\overline{q}}=\begin{cases}
        N & \text{ if } \overline{p}\leq \overline{q}. \\
        0 & \text{ else.} \\
    \end{cases}\]
    \end{proposition}
    \begin{proposition}
        If $(\mathcal{C}, \otimes_\mathcal{C}, [~,~], \mathcal{I})$ is a closed symmetric monoidal category then so is the category of perverse object $\mathcal{C}^{\pGM}$.
        The monoidal structure is given by:
        \[(M \boxtimes N)_{\overline{r}}:=\colim_{\overline{p}+\overline{q}\leq \overline{r}}M_{\overline{p}}\otimes_\mathcal{C} N_{\overline{q}}\]
        with unit $F_{\overline{0}}(\mathcal{I})$. The closed structure comes from
        \[\Hom_{\mathcal{C}^{\pGM}}(M,N)_{\overline{r}}:=\lim_{\overline{r}\leq \overline{q}-\overline{p}} [M_{\overline{p}},N_{\overline{q}}].\]
    \end{proposition}
    \begin{rem}
        Using adjunctions, we can also think of the closed structure as morphisms in the category of perverse objects (i.e. natural transformations)
        \[\Hom_{\mathcal{C}^{\pGM}}(M,N)_{\overline{r}}\simeq \mathcal{C}^{\pGM}(F_{\overline{r}}(\mathcal{I})\boxtimes M, N).\]
    \end{rem}
    
    \begin{rem}\label{rem:internal_tenhom_adj}
        Since $\mathcal{C}^{\pGM}$ is a symmetric monoidal category, for any perverse objects $X_\bullet, Y_\bullet$ and $Z_\bullet$ we have an isomorphism
        \[\Hom_{\mathcal{C}^{\pGM}}(X\boxtimes Y , Z)_\bullet \simeq \Hom_{\mathcal{C}^{\pGM}}(X , \Hom_{\mathcal{C}^{\pGM}}(Y,Z))_\bullet.\]
        We'll refer to this isomorphism as \emph{internal tensor-hom adjunction}.
    \end{rem}
    
    \begin{exemple}\label{ex:pmonoidal_chcplx}
    Let $\overline{r}$ be a perversity and $Z^\bullet,~Y^\bullet$ two perverse chain complexes. 
    The monoidal structure on $(Ch(R))^{\pGM}$ is given by 
    \[(Z \boxtimes Y)^{\overline{r}}:=\colim_{\overline{p}+\overline{q}\leq \overline{r}}Z^{\overline{p}}\otimes Y^{\overline{q}}\]
    where the tensor product on the right hand side denotes the tensor product of chain complexes.
    For $k\in \mathbb{Z}$, the degree $k$ and perverse degree $\overline{r}$ component of the perverse chain complex is
    \[(Z \boxtimes Y)_k^{\overline{r}}=\colim_{\overline{p}+\overline{q}\leq \overline{r}}\bigoplus_{i+j=k}Z_i^{\overline{p}}\otimes_R Y_j^{\overline{q}}.\]
    The unit is $F_{\overline{0}}(\mathbb{S}^0)^\bullet$.
    
    The internal Hom from $Z^\bullet$ to $Y^\bullet$ is the perverse chain complex whose degree $k$ and perverse degree $\overline{r}$ component is
    \[\Hom_{(Ch(R))^{\pGM}}(Z,Y)_k^{\overline{r}}=\lim_{\overline{r}\leq \overline{q}-\overline{p}} \prod_{j-i=k} \Hom_R(Z_i^{\overline{p}}, Y_j^{\overline{q}}).\]
    In what follows, we will be interested in the linear dual $DZ^\bullet$ of a perverse chain complex. We have:
    \[DZ_k^{\overline{r}}:=\Hom_{(Ch(R))^{\pGM}}(Z, F_{\overline{0}}(\mathbb{S}^0))_k^{\overline{r}}=\Hom_R(Z_{-k}^{\overline{t}-\overline{r}}, R).\]
    \end{exemple}
    \subsubsection{Perverse differential graded algebra}\label{sect:pDGA}
    We are now going to present a perverse analogue of differential graded algebras (DGA). Recall that a DGA is a monoid in the monoidal category $Ch(R)$, this justifies the following definition.
    \begin{defi}
        A \emph{perverse differential graded algebra} (pDGA) is a monoid in the category of perverse chain complexes $(Ch(R))^{\pGM}$.
    \end{defi}
    An equivalent explicit description is given in the next definition.
    \begin{defi}
        A \emph{perverse differential graded $R$-algebra} $A^\bullet$ is a perverse chain complex $(A^\bullet, d^\bullet)$ 
        equipped for every $i,j\in\mathbb{Z}$ and $\overline{p},\overline{q}\in \pGM$ (such that $\overline{p}+\overline{q}\leq \overline{t}$) with an associative product $\mu: A_i^{\overline{p}} \otimes A_j^{\overline{q}} \to A_{i+j}^{\overline{p}\oplus\overline{q}}$
        compatible with the poset structure of $\pGM$ i.e. it makes the following diagram commute
        \[\begin{tikzcd}
        A^{\overline{p_1}}_i\otimes A^{\overline{q_1}}_j \arrow[r, "\mu"] \arrow[d, "\phi_{\overline{p_1}\leq\overline{p_2}}\otimes \phi_{\overline{q_1}\leq\overline{q_2}}"'] & A^{\overline{p_1}\oplus\overline{q_1}}_{i+j} \arrow[d, "\phi_{\overline{p_1}\oplus\overline{q_1}\leq\overline{p_2}\oplus\overline{q_2}}"] \\
        A^{\overline{p_2}}_i\otimes A^{\overline{q_2}}_j \arrow[r, "\mu"]                                                                                                      & A^{\overline{p_2}\oplus\overline{q_2}}_{i+j}  \end{tikzcd}\]
        where $\overline{p_1}\leq \overline{p_2}$ and $\overline{q_1}\leq \overline{q_2}$.
        We will denote $\mu(a \otimes b)$ by $ab$. The multiplication has a unit $1\in A^{\overline{0}}_0$. Furthermore, the differential $d^\bullet$ is a derivation with respect to this product i.e.
        \[d_{i+j}^{\overline{p}\oplus \overline{q}}(ab)=(d_i^{\overline{p}}a)b+(-1)^ia(d_j^{\overline{q}}b) \text{ for } a\in A_i^{\overline{p}}\text{ and }b\in A_j^{\overline{q}}.\]
        We will denote by $\pdgAlg$ the category of perverse differential graded algebras whose morphisms are natural transformations between perverse chain complexes that are compatible with the products.\\
        We say that $(A^\bullet, d^\bullet )$ is a \emph{perverse commutative differential graded $R$-algebra} (pCDGA) if for every $\overline{p},\overline{q}\in \pGM$ and $i,j\in\mathbb{Z}$, the following diagram commutes
        \[\begin{tikzcd}
        A^{\overline{p}}_i\otimes A^{\overline{q}}_j \arrow[rr, "\tau"'] \arrow[rd, "\mu"'] &                                     & A^{\overline{q}}_j\otimes A^{\overline{p}}_i \arrow[ld, "\mu"] \\
                                                                                    & A^{\overline{p}\oplus\overline{q}}_{i+j}. &                                                               
        \end{tikzcd}\]
        where $\tau$ is the usual twisting isomorphism for chain complexes
        \[\begin{array}{rrcl}
        \tau: & A_\ast \otimes B_\ast & \to & B_\ast\otimes A_\ast \\
        ~& a\otimes b & \mapsto & (-1)^{\Real{a}\Real{b}}b\otimes a. \\
        \end{array}\]    
    \end{defi}
    \begin{rem}
        Note that having a perverse graded algebra is equivalent to asking for chain maps, $\mu: (A\boxtimes A)^{\overline{r}}\to A^{\overline{r}}$, for every $\overline{r}\in\pGM$, that are compatible with the poset structure of $\pGM$ and verify the associativity property.
    \end{rem}
    \begin{rem}
        To be precise, we have defined a perverse differential \emph{lower} graded $R$-algebra since we consider chain complexes. Similarly, we can define a perverse differential \emph{upper} graded $R$-algebra when we work on cochain complexes.
    \end{rem}
    \begin{exemple}
        Let $X$ be a filtered space. The blow-up of Sullivan's polynomial forms $\widetilde{A}^{\ast}_{PL,\overline{p}}(X)$ is a pCDGA, see \cite[Section 2.1]{CST18ration}.
    \end{exemple}
    \begin{exemple}\label{ex:tensoralg}
        Let $M^\bullet$ be a perverse graded module. The \emph{tensor algebra} of $M^\bullet$ is the perverse differential graded algebra \[(TM)^\bullet:= \bigoplus_{k\geq 0} (T^k M)^\bullet\]
        with $(T^0 M)^\bullet = F_{\overline{0}}(\mathbb{S}^0)^\bullet$ and $(T^k M)^\bullet=(M^{\boxtimes k})^\bullet$ for $k\geq 1$. 
        
        The degree of an element $m=m_1\otimes m_2 \otimes \ldots \otimes m_k$ in $(T^k M)^\bullet_\ast$ is $\Real{m}=\sum_{i=1}^k \Real{m_i}$ and its differential is given by
        $d(m_1\otimes \ldots \otimes m_k)=\sum_{i=1}^k (-1)^{\eps_i}m_1\otimes \ldots d(m_i)\ldots \otimes m_k$
        where $\eps_i=\sum_{j<i} \Real{m_j}$. 
        
        For $\overline{p},\overline{q}\in \pGM$ and $i,j\in\mathbb{Z}$, the multiplication 
        \[\mu: (TM)_i^{\overline{p}} \otimes (TM)_j^{\overline{q}} \to (TM)_{i+j}^{\overline{p}\oplus\overline{q}}\] is given by the tensor product.
    \end{exemple}
    By a similar process, we adapt the notion of differential graded $A$-module to the context of perverse objects.
    \begin{defi}
        Let $(A^\bullet, d_A^\bullet)$ be a perverse DGA. A left \emph{perverse differential graded $(A^\bullet, d^\bullet_A)$-module} (pDG module) is a perverse chain complex of $R$-modules $(M^\bullet, d^\bullet_M)$ with $R$-linear maps for every $\overline{p}, \overline{q}\in \pGM$ (such that $\overline{p}+\overline{q}\leq \overline{t}$) and $i, j\in \mathbb{Z}$
        \[\begin{array}{rcl}
         A^{\overline{p}}_i\otimes M^{\overline{q}}_j    & \to & M^{\overline{p}\oplus \overline{q}}_{i+j} \\
           a \otimes m  & \mapsto & a.m
        \end{array} \]
        which verify 
        \[d_M(a.m)=d_A(a).m+(-1)^{\Real{a}}a.d_M(m).\] 
        The category of perverse differential graded $A^\bullet$-modules is denoted $\pMod(A)$.
   
    \end{defi}
    \begin{rem}
    \leavevmode
        \begin{itemize}
        \item Analogously, one defines \emph{right perverse differential graded $(A^\bullet, d^\bullet_A)$-module}.
        \item The \emph{enveloping algebra} of a pDGA $(A^\bullet, d_A^\bullet)$ is given by $((A^e)^\bullet:=(A\boxtimes A^{op})^\bullet, d^\bullet_A \otimes 1 +1 \otimes d^\bullet_A)$ where $(A^{op})^\bullet$ is equipped with the opposite multiplication of $A^\bullet$, i.e.\ $\mu^{op}(a\otimes b)=(-1)^{\Real{a}.\Real{b}}\mu(b\otimes a)$. A \emph{perverse differential graded $(A^\bullet, d^\bullet_A)$-bimodule} is a left perverse differential graded $(A^e)^\bullet$-module. It is a perverse chain complex of $R$-modules $(M^\bullet, d^\bullet_M)$ with left and right pDG $A^\bullet$-module structures such that
        \[d_M(a.m.b)=d_A(a).m.b+(-1)^{\Real{a}}a.d_M(m).b+(-1)^{\Real{a}+\Real{m}}a.m.d_B(b)\]
        for any $a,b\in A^\bullet$ and $m\in M^\bullet.$
        We denote by $\pMod((A^e)^\bullet)$ the category of perverse differential $A^\bullet$-bimodules.
        \end{itemize}    
    \end{rem}
    The category $\pMod(A)$ has an internal Hom functor which we define below.
    \begin{defi}
    Let $A^\bullet$ be a pDGA and consider two pDG $A^\bullet$-modules $M^\bullet, P^\bullet$. The perverse chain complex $\Hom_A(M,P)^\bullet$ is the perverse subcomplex of the internal Hom, $\Hom_{(Ch(R))^{\pGM}}(M, P)^\bullet$, made of maps that commute with the action of $A^\bullet$. In other words, a map $f : M^\bullet \to P^\bullet$ of degree $\Real{f}$ of $\Hom_{(Ch(R))^{\pGM}}(M, P)^\bullet$ is in $\Hom_A(M,P)^\bullet$, if 
    \[f(a.m) = (-1)^{\Real{f}\Real{a}}a.f(m)~ \forall a \in A^\bullet,\, m\in M^\bullet.\]
    Similarly, we define $(M \boxtimes_A P)^\bullet$ the tensor product over $A^\bullet$ when $M^\bullet$ (resp. $P^\bullet$) is a right (resp. left) pDG $A^\bullet$-module. It is generated by the simple tensors $m\otimes p$ of $M\boxtimes P$ such that
    \[m.a\otimes p = m\otimes a.p ~\forall a \in A^\bullet.\]
    In other words, $(M\boxtimes_A P)^\bullet$ is the coequalizer of $f,g:(M\boxtimes A \boxtimes P)^\bullet \to (M\boxtimes P)^\bullet$ where $f(m\otimes a \otimes p)=m.a\otimes p$ and $g(m\otimes a \otimes p)=m\otimes a.p$. 
    \end{defi}

    \subsubsection{Properties of the blown-up cochain complex}
    Using the formalism we have exposed so far, we restate the results given at the end of subsection \ref{subsect:bupinter} concerning algebraic structures found on the blown-up intersection complex.
    \begin{proposition}
        Let $X$ be a stratified space. The blown-up intersection complex $(\widetilde{N}_\bullet^\ast(X; R), \cup)$ is a perverse differential (upper) graded algebra.\\
        The intersection chain complex $I^\bullet C_\ast(X;R)$ is a right perverse differential graded $\widetilde{N}_\bullet^\ast(X; R)$-module for the cap product. 
    \end{proposition}

    The next result follows from Poincaré duality (Theorem \ref{thm:Poincaré}). 
    \begin{proposition}
    Let $X$ be an $n$-dimensional, second countable, compact and oriented pseudomanifold. Then, there exists a quasi-isomorphism of perverse (right) $\widetilde N^\ast_{\overline{\bullet}}(X;R)$-modules:
    \[\widetilde N^\ast_{\bullet}(X;R) \to \Hom(I_{\bullet}C^{n-\ast}(X;R), R).\]
    It is obtained as the composite of
    \[\widetilde N^\ast_{\bullet}(X;R) \xrightarrow[\simeq]{DP_X} I^{\bullet}C_{n-\ast}(X;R) \xrightarrow[\simeq]{Bid} \Hom(I_{\bullet}C^{n-\ast}(X;R), R)\]
    where
    \begin{itemize}
        \item $DP_X$ is the cap product with a fundamental cycle
        \item and $Bid$ is the injection in the bidual given in \emph{\cite[Proposition A]{CST20Deligne}}.
    \end{itemize}
    \end{proposition}

    
    Recall that, when we work on a field, we have a quasi-isomorphism between the blown-up cochain complex and the dual intersection cochain complex (Theorem \ref{thm:field_intercoho_iso}). The previous result can be stated in the following terms.
    
    \begin{proposition}\label{prop:dual-qiso-bup}
    Let $X$ be an $n$-dimensional, compact, second countable and oriented pseudomanifold. Then, there exists a quasi-isomorphism of perverse (right) $\widetilde N^\ast_{\overline{\bullet}}(X;\mathbb{F})$-modules:
    \[
    \Dual_N:\widetilde N^\ast_{\bullet}(X;\mathbb{F}) \to \mathbb{D}\widetilde N^\bullet_\ast(X;\mathbb{F})\]
    where $\mathbb{D}\widetilde N^\bullet_\ast(X;\mathbb{F}):=\Hom(\widetilde N^{n-\ast}_{\overline{t}-\bullet}(X;\mathbb{F}), \mathbb{F})$.
    \end{proposition}
    Note that $\mathbb{D}\widetilde N^\bullet_\ast(X;\mathbb{F})$ is equal to the linear dual $D\widetilde N^\bullet_\ast(X;\mathbb{F})$ by a shift of $-n$ in degree (Example \ref{ex:pmonoidal_chcplx}).
    The image of the unit $1\in\widetilde N^0_{\overline{0}}(X; \mathbb{F})$ by $\Dual_N$ is denoted $\Gamma_X$, it is given by 
    \begin{align*}
        \Gamma_X: \widetilde N_{\overline{t}}^n(X;\mathbb{F}) &\to \mathbb{F}   \\
          \alpha &\mapsto \Inter(1)(\alpha \cap \zeta)
    \end{align*}
    with $\zeta\in I^{\overline{0}}C_n(X;R)$ a fundamental cycle.

\emergencystretch=1em 
\printbibliography[title={Bibliography}] 
\end{document}